\documentclass[11pt,a4paper]{amsart}
\usepackage[a4paper]{geometry} 
\usepackage[english]{babel}
\usepackage{microtype}
\usepackage{mathtools}
\usepackage{amssymb,amsthm}
\usepackage{thmtools}
\usepackage{subcaption}
\usepackage{graphicx}
\usepackage[usenames,dvipsnames]{xcolor}
\usepackage[foot]{amsaddr}
\usepackage{etoolbox}
\usepackage{appendix}

\usepackage[
final,
colorlinks,
linkcolor = BrickRed,
citecolor = OliveGreen,
filecolor = black,
urlcolor = BrickRed,
bookmarks=true,
bookmarksnumbered=true,
bookmarksopen=true,
bookmarksopenlevel = 0,
hypertexnames=false,
]{hyperref}
\usepackage{url}

\usepackage{algorithm}
\makeatletter
\def\plist@algorithm{Alg.\space}
\makeatother
\usepackage{algpseudocodex}
\algrenewcommand\algorithmicrequire{\textbf{Input:}}
\algrenewcommand\algorithmicensure{\textbf{Output:}}
\usepackage{csquotes}

\newtheorem{theorem}{Theorem}
\newtheorem{lemma}[theorem]{Lemma}
\newtheorem{corollary}[theorem]{Corollary}
\theoremstyle{definition}
\newtheorem{definition}[theorem]{Definition}
\theoremstyle{remark}
\newtheorem{remark}[theorem]{Remark}
\numberwithin{equation}{section}
\usepackage[capitalise,nameinlink,noabbrev]{cleveref}

\AddToHook{cmd/appendix/before}{\crefalias{section}{appendix}}

\newcommand{\tol}{\mathtt{tol}}
\newcommand{\polylog}{\operatorname{polylog}}
\newcommand{\card}{\operatorname{card}}
\newcommand{\supp}{\operatorname{supp}}
\newcommand{\G}{\mathcal{G}}
\newcommand{\N}{\mathcal{N}}
\newcommand{\cS}{\mathcal{S}}

\newcommand{\Colm}{C_{\mathrm{ol},\ell}}
\newcommand{\nei}{\mathsf{N}}
\newcommand{\s}{\mathfrak{a}}
\newcommand{\uloc}{\varphi}
\newcommand{\qq}{\tilde{q}}

\usepackage[
    backend=biber,
    style=alphabetic,
    sorting=nyt,
    giveninits=true,
    isbn=false,
    eprint=true,
    maxnames=10
]{biblatex}
\usepackage{tikz}
\usetikzlibrary{patterns}
\usepackage{pgfplots}
\usetikzlibrary{positioning, calc, backgrounds}
\usepgfplotslibrary{statistics} 

\definecolor{unia-green}{RGB}{0, 101, 97}
\definecolor{unia-pink}{RGB}{173, 0, 124}
\definecolor{unia-yellow}{RGB}{246, 168, 0}
\definecolor{unia-orange}{RGB}{235, 105, 11}
\definecolor{unia-red}{RGB}{212, 0, 45}
\definecolor{unia-lightblue}{RGB}{0, 174, 207}
\definecolor{unia-blue}{RGB}{0, 135, 193}
\definecolor{unia-lightgreen}{RGB}{72, 147, 36}

\usepackage{booktabs}

\title{Quantum Enhanced Numerical Homogenization}

\author[]{L.~Balazi$^{*}$, M.~Deiml$^{*}$, D.~Peterseim$^{\dagger}$}
\address{${}^{*}$ Institute of Mathematics, University of Augsburg, Universit\"atsstr.~12a, 86159 Augsburg, Germany}
\address{${}^{\dagger}$ Institute of Mathematics \& Centre for Advanced Analytics and Predictive Sciences (CAAPS), University of Augsburg, Universit\"atsstr.~12a, 86159 Augsburg, Germany}
\email{\{loic.balazi, matthias.deiml, daniel.peterseim\}@uni-a.de}

\thanks{Funded by the Deutsche Forschungsgemeinschaft (DFG, German Research Foundation) -- 571768116.}
\date{\today}

\begin{document}

\begin{abstract}
We propose a numerical homogenization method for scalar linear partial differential equations with rough coefficients that integrates classical coarse-scale solvers with quantum subroutines for fine-scale corrections. Inspired by the Localized Orthogonal Decomposition, we employ quantum local problem solvers to capture fine-scale features efficiently. Unlike periodic homogenization approaches, it does not rely on any periodicity assumption. Moreover, the coupling between quantum computation and the coarse model requires only selected measurements of quantum representative volume elements, thereby mitigating the quantum-interface information bottleneck that could otherwise negate a potential speed-up. We show that the local quantum solver can achieve solutions with the required level of accuracy with an operation count that scales only logarithmically with the fine-scale resolution, as determined by the smallest length scale encoded in the diffusion coefficient. The potential of the approach is illustrated through two-dimensional numerical experiments, using a classical simulation of the local quantum solver.
\end{abstract}

\maketitle

%=========  Key words and AMS subject classifications
\indent
{\tiny {\bf AMS subject classifications.} 
{\bf 65N30}, %Finite element, Rayleigh-Ritz and Galerkin methods for boundary value problems involving PDEs
{\bf 68Q12}, %Quantum algorithms and complexity in the theory of computing
{\bf 35J15} %Second-order elliptic equations
}

\section{Introduction}
Consider a prototypical scalar linear elliptic problem with a rough, highly oscillatory diffusion coefficient, possibly including lower-order terms, in a domain $\Omega \subset \mathbb{R}^d$,
\begin{equation*}
 -{\rm div} (\mathfrak{A}_\varepsilon \nabla u) + \mathfrak{b} \cdot \nabla u + \mathfrak{c} u =  f,
\end{equation*}
where the diffusion coefficient $\mathfrak{A}_\varepsilon$ encodes microscopic features, 
with $\varepsilon$ denoting a characteristic length scale or, more generally, the smallest scale in a continuum of scales. It is well known that achieving accurate solutions requires sufficiently fine spatial discretizations, leading to linear systems of size $\mathcal{O}(\varepsilon^{-d})$. For small $\varepsilon$, the problem quickly becomes intractable, as the computational demands exceed available resources. In this context, multi-scale methods or numerical homogenization (see~\cite{Altmann21} for a detailed overview), including the (Generalized) Multi-scale Finite Element Method~(MsFEM)~\cite{Hou97,Efendiev09,Efendiev13,Blanc23,Chung23}, the Localized Orthogonal Decomposition (LOD)~\cite{Malqvist14,Henning14,Malqvist21}, Operator-Adapted Wavelets~\cite{Owhadi_Scovel_2019}, Fast Fourier Transform (FFT)-based homogenization methods~\cite{Moulinec1994,Moulinec1998,Schneider21}, among many others, have been developed in the literature. A common feature of these approaches is that under-resolved fine-scale information is recovered through the solution of local cell, corrector, or RVE-type problems. Yet these local fine-scale computations can themselves become a computational bottleneck, particularly in the absence of periodicity or in parameter-dependent settings that require the evaluation of families of local corrector functions. Recently, several works have proposed data-driven approaches to overcome these limitations~\cite{Kropfl22,Han23,Stepanov23,Strom25}, while emerging quantum computing paradigms offer a promising new direction for addressing such problems~\cite{Ruane2025QuantumIndex}. 

Quantum computing offers significant opportunities for computational mechanics, although practical applications in this field remain at an early stage. Existing works have explored quantum approaches for solving the Poisson problem~\cite{Cao13,Childs21,Vazquez22,Deiml25}, non-stationary transient problems such as the wave equation~\cite{Costa19}, and the heat equation \cite{Linden22}. For an overview of early developments in quantum algorithms for partial differential equations (PDEs) in structural mechanics, we refer to~\cite{Tosti22}. Complementing these developments, the feasibility of implementing domain decomposition preconditioners within a quantum framework has recently been demonstrated in~\cite{fressart26}.  Also, in recent work, Schr\"odingerization has been proposed as a general framework for quantum simulation of linear ordinary and partial differential equations, by transforming the underlying equations into Schr\"odinger-type systems through a warped phase transformation~\cite{Jin23,Jin24}, with applications to multiscale PDEs also being investigated~\cite{Hu24}. Lately, quantum-based homogenization methods in computational mechanics have started to be investigated, motivated by the need to reduce the computational complexity of the RVE computations. For example,~\cite{Liu24,Wang26} present a quantum computing framework for solving RVE problems in computational homogenization by combining conventional algorithms, including fixed-point iteration and the Fast Fourier Transform, reformulated into their quantum counterparts. The authors show that a quantum state representing the solution can be prepared in a number of operations that is polylogarithmic in the number of degrees of freedom. However, since the method relies on Fourier-based formulations, it is inherently restricted to periodic homogenization problems.
A recent work~\cite{Xu26} proposes using quantum computing to solve RVE problems obtained from classical homogenization. The method formulates the local problem on a periodic RVE, leading to a linear system $\mathbf{K}u=\mathbf{f}$ derived from finite element discretization with periodic boundary conditions. The approach then extracts quantities of interest from the quantum state to approximate the homogenized coefficient.

In this paper, we propose a hybrid multiscale approach that exploits quantum speed-ups for the most expensive local sub-tasks while retaining a classical coarse-scale formulation for the global problem. A central motivation is to go beyond existing quantum-based computational homogenization approaches, which are primarily based on quantum formulations of periodic RVE problems. In contrast, our approach does not require periodicity and is applicable to non-periodic, heterogeneous, and potentially stochastic settings.
A further motivation comes from a common perspective in computational mechanics: one is often not interested in the full fine-scale solution field itself, but rather in macroscopic quantities of interest derived from it, such as compliance in structural mechanics, heat fluxes through interfaces in heat transfer, or lift and drag in fluid mechanics. This viewpoint is particularly natural for quantum solvers, where the PDE solution is encoded as a quantum state, and the full reconstruction of the fine-scale solution may require a prohibitive number of measurements.

Inspired by the Localized Orthogonal Decomposition (LOD) method, our hybrid strategy replaces the classical solution of localized fine-scale corrector problems with quantum local solvers. The coupling of these quantum computations with the classical coarse model requires only selected measurements of the quantum local solutions. Crucially, the number of measurements required for each local problem is relatively small and independent of the fine and coarse discretization sizes. This avoids the quantum-interface information bottleneck that could otherwise eliminate the potential speed-up. In addition, the proposed formulation avoids another common difficulty in quantum PDE solvers: the encoding of arbitrary right-hand sides as quantum states. In our setting, each local problem has a prescribed source term determined by the coarse basis functions and the local coefficient. These source terms are explicit and provably encodable, thereby avoiding the need to encode general right-hand sides, which is known to be difficult in general~\cite[Appendix A]{Deiml25}.

In our approach, the finite element discretization of the local problems leads to linear systems of the form $\mathbf{K}_{\rm loc} u = \mathbf{g}$, where $\mathbf{g}$ is a simple right-hand side that is independent of the global source term. The solution of such systems by quantum algorithms can build on recent advances in quantum linear system methods; see, e.g.,~\cite{Harrow09,Ambainis12,Childs17,Gilyen19,Subasi19,Lin20,An22,Lin22,Low24}. In this work, we solve the local problems using the preconditioned approach developed in \cite{Deiml25}, combined with a discretization that yields a computational complexity that is logarithmic in the local discretization level and depends only linearly on the inverse measurement tolerance. It should be emphasized that a naive implementation would increase the computational complexity by a factor proportional to the condition number of $\mathbf{K}_{\rm loc}$.

The role of \cite{Deiml25} in the present work should be distinguished from a direct quantum solution of the full PDE. Indeed, the method of \cite{Deiml25} could, in principle, be applied directly to the full problem, with a similar dependence on the fine-scale resolution and without structural assumptions on the coefficients, apart from encodability. Our hybrid approach instead uses the quantum solver only in an offline stage to compute selected local quantities. These quantities are then transferred to a classical coarse-scale model, which can be reused in the online stage to obtain coarse-mesh approximations for arbitrary coarse-scale right-hand sides at a cost essentially independent of the fine-scale resolution and without further calls to the quantum subroutines.
\medskip

After introducing the model problem in \cref{sec.Model}, we present the proposed approach in~\cref{sec.method}. We then describe the discretization strategy and the quantum solution of the local problems in \cref{sec.Quantum}. In \cref{sec.Sens}, we perform a sensitivity analysis to assess the impact of stochastic measurement errors on the overall simulation. Finally, in \cref{sec.num}, we evaluate this approach through two-dimensional numerical experiments for both periodic and non-periodic highly oscillatory diffusion coefficients, using a classical simulation of the quantum solver, and discuss the results.

\section{Model problem}
\label{sec.Model}
Let $\Omega\subset \mathbb R^d$, for $d\in\{1,2,3\}$, be a bounded domain with  Lipschitz boundary $\partial\Omega$. We consider the following problem
\begin{equation}
\begin{aligned}
 -\nabla\cdot(\mathfrak{A} \nabla u) + \mathfrak{b} \cdot \nabla u + \mathfrak{c} u &&=&& & f &&\text{in } \Omega, \\
 u && = && & 0 && \text{on } \partial \Omega,
\end{aligned}
\label{eq:Problem}
\end{equation}
where $f$ is an external force, $\mathfrak{c} \in L^\infty (\Omega)$, $\mathfrak{b} \in L^\infty(\Omega;\mathbb{R}^d)$ with  $\mathrm{div}(\mathfrak{b}) \in L^\infty (\Omega)$,
and ${\mathfrak{A} \in L^\infty (\Omega; \mathbb{R}^{d\times d})}$ with uniform spectral bounds (uniformly elliptic), i.e., 
	\begin{equation*}
	\left\{
	\begin{aligned}
	0<\theta&:= \mathrm{ess}\inf \limits_{x\in \Omega} \inf \limits_{z\in \mathbb{R}^d\setminus \{0\}}\frac{z^\top\mathfrak{A}(x)z}{z\cdot z}\\
	&\leq \mathrm{ess}\sup \limits_{x\in \Omega} \sup \limits_{z\in \mathbb{R}^d\setminus \{0\}}\frac{z^\top\mathfrak{A}(x)z}{z\cdot z}=:\Theta <\infty.
	\end{aligned}
	\qquad\right.  
%\label{eq:UnifSpectBound}
	\end{equation*}
We further assume that the lower-order coefficients satisfy
\begin{equation*}
-\frac{1}{2}\mathrm{div}(\mathfrak{b})+\mathfrak{c}\geq 0
\quad\text{a.e. in }\Omega .
%\label{eq:lower_order_assumption}
\end{equation*}
For a domain $\omega \subseteq \Omega$, we denote by $(\bullet,\bullet)_\omega$ the usual scalar product in $L^2(\omega)$, and consider $(\bullet,\bullet):=(\bullet,\bullet)_\Omega$. Let us define the space $V(\omega):=\{v \in H^1(\Omega) \ {\rm such \ that} \ {\rm supp}(v) \subseteq \omega\}$ equipped with the classical $H^1$ semi-norm, defined as  $\lVert \, \bullet \, \rVert_ {V, \omega} := \sqrt{ \int_\omega \lvert \nabla \bullet \rvert^2 \mathrm{d}x}$. For the ease of notation, we denote $V:= V(\Omega) =  H^1_0(\Omega)$ and  $\lVert \,  \bullet \,  \rVert_ {V} := \lVert \,  \bullet \,  \rVert_ {V,\Omega} $. We denote by $\langle \bullet, \bullet \rangle_{V'\times V}$ the duality pairing between $V$ and its dual space $V':=H^{-1}(\Omega)$. Let us define the bilinear form $a: V \times V \to \mathbb{R}$ as
\begin{equation*}
a(v,w):= (\mathfrak{A}\nabla v,\nabla w)
        +(\mathfrak{b}\cdot\nabla v,w)
        +(\mathfrak{c} v,w),
\end{equation*}
and let us consider $F \in V'$. Then, a weak formulation of the problem \eqref{eq:Problem} reads as follows. Seek $u\in V$ such that, for all $v \in V$,
\begin{equation}\label{e:ProblemWeak}
a(u,v) = F(v).
\end{equation}
Under the assumptions above, the bilinear form $a$ is coercive and continuous on $V$. In particular, there exist constants $\alpha,\beta>0$ such that
\begin{equation}
\alpha \lVert v \rVert_V^2 \leq a(v,v)
\quad\text{for all } v\in V,
\qquad
|a(v,w)| \leq \beta \lVert v \rVert_V \lVert w \rVert_V
\quad\text{for all } v,w\in V.
\label{eq:abound}
\end{equation}
Consequently, the Lax--Milgram theorem implies that the variational problem~\eqref{e:ProblemWeak} is well posed; that is, it admits a unique solution $u \in V$ depending continuously on the data.

\section{Definition of the method}
\label{sec.method}
In this section, we briefly recall the main idea of the Localized Orthogonal Decomposition (LOD) method, which has been extensively studied; see, e.g.,~\cite{Henning15,Malqvist15,Gallistl15,Gallist17,Altmann21}. The aim is not to provide a detailed analysis of the method, but rather to review the underlying concepts and explain why the LOD framework is particularly natural for the hybrid quantum--classical approach developed in this work.

\subsection{Meshes and data structures}
Let $\G_H$ be a regular partition of $\Omega$ into closed intervals, parallelograms, and parallelepipeds for $d=1,2,3$, respectively, such that $\bigcup_{T\in\G_H}T =\overline\Omega$ and any two distinct $T,T'\in\G_H$ are either disjoint or intersect in a common vertex, edge, or face, depending on the dimension. We impose shape-regularity in the sense that the aspect ratios of the elements in $\G_H$ are uniformly bounded. Since we consider quadrilaterals, respectively hexahedra, with parallel faces, this guarantees the non-degeneracy of the elements in $\G_H$. The global mesh size is defined by
$H:=\max\{\operatorname{diam}(T) \ {\rm for} \ T\in\G_H\}$. The method described in this section carries over to simplicial triangulations, to more general quadrilateral or hexahedral partitions satisfying suitable non-degeneracy conditions, and even to meshless methods based on suitable partitions of unity \cite{Henning15}.

Given any subdomain $\omega\subseteq\overline\Omega$, define its neighborhood via
\begin{equation*}
\nei(\omega):=\operatorname{int}
          \big(\cup\{T\in\G_H \ {\rm such \ that} \ T\cap\overline{\omega}\neq\emptyset  \}\big).
\end{equation*}
Furthermore, we define, for any $\ell\geq 2$, the patches
\begin{equation}
\nei^1(\omega):=\nei(\omega)
\qquad\text{and}\qquad
\nei^\ell(\omega):=\nei(\nei^{\ell-1}(\omega)) .
\label{eq:Patch}
\end{equation}
The shape-regularity of the mesh implies that there is a uniform bound 
$\Colm$ on the number of elements in the $\ell$th-order patch,
\begin{equation*}
\max_{T\in\G_H}\card\{ K\in\G_H\ {\rm such \ that} \ K\subseteq \overline{\nei^\ell(T)}\}
\leq \Colm.
\end{equation*}
Throughout this paper, we assume that the coarse-scale mesh $\G_H$
is quasi-uniform. This implies that $\Colm$ depends polynomially
on $\ell$. Also, we write $a \lesssim b$ if there exists a generic constant $C > 0$, independent of the mesh size and $\ell$, such that $a \leq C b$, and we write $a \approx b$ if $a \lesssim b$ and $b \lesssim a$.

\subsection{Quantities of interest and fine-scale space}
The LOD method is based on splitting the solution $u$ into a coarse scale contribution and a fine-scale part. We tackle this task from a more abstract perspective and define a number of macroscopic quantities of interest that extract the desired information from the exact solution. These linear functionals are denoted by $q_j \in V'$, $j \in \mathcal{J}$, for some finite index set $\mathcal{J}$ of size $N :=\lvert \mathcal{J} \rvert$. We assume that these quantities are linearly independent. The idea of LOD is to construct a discrete method that approximates the quantities $q_j(u)$ as closely as possible. Several examples of macroscopic quantities of interest $q_j$, including local integral means of functions over suitable entities of some finite element mesh or local averaging, are discussed in \cite{Altmann21}.

Given the macroscopic quantities of interest, we define the fine-scale space as 
\begin{equation*}
    W:= \{ v \in V \ {\rm such \ that} \ \forall j \in \mathcal{J}, \  q_j(v)=0 \} = \bigcap_{j \in \mathcal{J}} \ker(q_j).
\end{equation*}
The space $W$ is closed since it is the intersection of finitely many kernels of continuous linear functionals. We now define the fine-scale projection operator, also called the corrector, $\mathcal{C}^*:V \to W$ by
\begin{equation}
    a(w, \mathcal{C}^* v) = a(w,v)\quad \text{for all } w \in W,\ v \in V.\label{eq:c1}
\end{equation}
Equivalently, 
\begin{equation*}
    a(w, (1-\mathcal{C}^*) v) = 0\quad \text{for all } w \in W,\ v \in V.
\end{equation*}
For $v\in W$, the definition implies $\mathcal{C}^*v=v$. Hence, $\mathcal{C}^*$ is a projection onto $W$ and
\begin{equation*}
    W=\operatorname{range}(\mathcal{C}^*)=\ker(1-\mathcal{C}^*).
\end{equation*}
Based on the kernel space $W$ and the projection $\mathcal{C}^*$, we construct the (finite-dimensional) test space 
\begin{equation*}
 \widetilde{V}_H :=(1 - \mathcal{C}^*)V.
\end{equation*}
Note that an analogous fine-scale projection $\mathcal{C}:V \to W$ associated with the primal problem is defined by
\begin{equation}
    a( \mathcal{C} v, w) = a(v, w)\quad \text{for all } v \in V,\ w \in W.
\label{eq:Cprimal}
\end{equation}
The tilde indicates that oscillatory fine-scale features have been incorporated into the coarse space. The space $\widetilde{V}_H$ has dimension $N$ and gives rise to a stable decomposition of $V$.
\begin{theorem}[$a$-\enquote{orthogonal} decomposition, Theorem 3.5 \cite{Altmann21}]
Let the quantities of interest $q_j \in V'$, $j \in \mathcal{J}$, be pairwise linearly independent. Under the assumption of well-posedness of~\eqref{eq:c1}, the space $\widetilde{V}_H$ has dimension $N=\lvert \mathcal{J} \rvert$ and defines a conforming decomposition of the overall space, namely
    \begin{equation*}
        V = \widetilde{V}_H \oplus W.
    \end{equation*}
Furthermore, we have the \enquote{orthogonality} relation
\begin{equation*}
    a(W, \widetilde{V}_H)=0. 
\end{equation*}
\label{thm:decomposition}
\end{theorem}
In \cref{thm:decomposition}, the word orthogonal is in quotes since the bilinear form $a$ does not define a scalar product. Below, we propose two other characterizations of the coarse space $\widetilde{V}_H$. The first one is
\begin{equation*}
    \widetilde{V}_H = \{ v \in V \ \text{such that, for all} \ w \in W, \ a (w, v) =0 \}. 
\end{equation*}
For the other characterization, let $\mathcal{L}^*: V \to V'$ be the linear operator associated with the bilinear form $a$ by 
\begin{equation}
  \langle \mathcal{L}^*v, w \rangle_{V' \times V} :=a(w,v) \quad {\rm for \ all}\  v, w \in V.
\label{eq:L}
\end{equation}
The boundedness of $a$ guarantees that $\mathcal{L}^*$ is well defined, and its invertibility follows from the well-posedness of the adjoint problem. We denote by $\mathcal{L}^{-*}:=(\mathcal{L}^*)^{-1}$ its inverse. Then, a less obvious characterization of the discrete space $\widetilde{V}_H$, which will be important below, is given by 
\begin{equation*}
    \widetilde{V}_H = {\rm span}\{ \mathcal{L}^{-*} q_j, \ j \in \mathcal{J} \}.
\end{equation*}
To see this, first note that both spaces have the same dimension $N$. 
Moreover, a function $u_j=\mathcal{L}^{-*} q_j$ satisfies, for all $w \in W$,
\begin{equation*}
    a(w,u_j) = a(w, \mathcal{L}^{-*} q_j) = q_j(w) =0,
\end{equation*}
which implies that
\[
{\rm span}\{\mathcal{L}^{-*} q_j, \ j \in \mathcal{J}\}
\subseteq \widetilde{V}_H .
\]
The equality then follows from the dimension argument. Note, however, that the basis functions $\mathcal{L}^{-*} q_j$ generally have global support and do not possess computationally useful decay properties for arbitrary macroscopic quantities of interest. For the implementation, the key step in numerical homogenization is therefore to replace these global basis functions with localized approximations. This localization procedure will be discussed in the following sections.

\subsection{Characterization of the coarse-scale space}
As mentioned earlier, the construction of a local basis of the multi-scale space $\widetilde{V}_H$ is crucial. The construction is based on a discrete subspace $V_H \subset V$ of dimension $N$ which has a local basis and satisfies $(1-\mathcal{C}^*)V_H = \widetilde{V}_H$. It turns out that the complementary correctors, namely~$(1 - \mathcal{C}^*)$, are quasi-local, so that a local basis of $V_H$ induces a local approximation basis of $\widetilde{V}_H$. Based on the partition $\mathcal{G}_H$, we define subsets on which the quantities of interest act. For this, we set $\omega^\prime_j$ as an element patch such that $q_j(v) = 0$ for all $v \in V$ with $\supp(v) \subseteq \Omega \setminus \omega^\prime_j$ and which is minimal with respect to the number of elements. In what follows, we assume that there is some (small) uniform parameter $r^\prime \in \mathbb{N}$, such that for each $j \in \mathcal{J}$, we can find a function~$\phi_j\in V$ with 

\begin{subequations}
\begin{align}
    \omega_j := \supp(\phi_j) &\subseteq \nei^{r^\prime}(\omega^\prime_j), \label{eq:Da} \\ 
    q_j(\phi_k) &= \delta_{kj} \quad {\rm for \ all} \ j,k \in \mathcal{J}, \label{eq:Db}
\end{align}%
\label{eq:D}%
\end{subequations}%
where $\delta$ is the Kronecker symbol. This means, in particular, that we assume the functions $\phi_j$ to be local in the sense that $\supp(\phi_j)$ is not much larger than $\omega^\prime_j$. Without loss of generality, we assume that $\omega_j$ is an element patch. In some sense, the functions $\phi_j$ are local $V$-conforming representatives of the quantities of interest~$q_j$. Now, let us define 
\begin{equation*}
    V_H := {\rm span} \{\phi_j, \ j \in \mathcal{J} \}, 
\end{equation*}
and 
\begin{equation*}
    \tilde{\phi}_j := (1- \mathcal{C}^*)  \phi_j \in \widetilde{V}_H, \ j \in \mathcal{J}. 
\end{equation*}
Then, it can be shown that $\tilde{\phi}_j$ form a basis of $\widetilde{V}_H$.  Also, the function $(1- \mathcal{C}^*) \phi_j$ does not depend on the
particular choice of the function $\phi_j$  as long as the properties \eqref{eq:D} are satisfied. Indeed, to see this, we characterize these functions in the form of a saddle point problem. For $j \in \mathcal{J}$, seek $\tilde{\phi}_j \in V$ and $\mu =[\mu_0, \dotsc, \mu_{N-1}]^\top \in \mathbb{R}^N$ such that
\begin{equation}
    \begin{array}{rll}
     \displaystyle  a(v,\tilde{\phi}_j) + \sum_{k\in \mathcal{J}} \mu_k q_k(v) &=0  & \forall v \in V  \\
       \displaystyle    q_k(\tilde{\phi}_j) &= q_k(\phi_j) = \delta_{kj} & \forall k \in \mathcal{J}.
    \end{array}
\label{eq:saddle1}
\end{equation}
This problem is well-posed due to condition \eqref{eq:Db}, the finite number of quantities of interest, and the inf-sup stability of $a$ on $W$ (the kernel of the constraints). Evidently, the unique solution~$(\tilde{\phi}_j, \mu)$ is characterized independently of $\phi_j$. Nonetheless, we show that $\tilde{\phi}_j = (1 - \mathcal{C}^*) \phi_j$ by proving that $( (1-\mathcal{C}^*)\phi_j,\mu)$ solves the saddle point problem \eqref{eq:saddle1} for some $\mu \in \mathbb{R}^N$. It holds that 
\begin{equation*}
    q_k((1-\mathcal{C}^*)\phi_j) = q_k(\phi_j) = \delta_{kj},
\end{equation*}
and $a(w, (1-\mathcal{C}^*)\phi_j)=0$ for all $w \in W$. 
Finally, we are able to characterize the Lagrange multipliers $\mu$ with the help of the remaining test functions $v \in V$ which are complementary to~$W$. 

\begin{lemma}[Characterization of the coarse space, Lemma 3.12 \cite{Altmann21}]
Let the corrector $\mathcal{C}^*$ be well defined. Given \eqref{eq:D}, the functions $\phi_j$, for $j\in \mathcal{J}$, are linearly independent such that the dimension of $V_H$ is equal to $N$. Furthermore, we have 
\begin{equation*}
    (1-\mathcal{C}^*) V_H = \widetilde{V}_H,
\end{equation*}
that is $(1-\mathcal{C}^*)$ is a bijection mapping from $V_H$ to $\widetilde{V}_H$.
\label{lemma:CS}
\end{lemma}
\Cref{lemma:CS} motivates the definition of a projection operator $P_H:V \to V_H$ using the linearly independent functions $\phi_j$, as
\begin{equation*}
    P_H v = \sum_{j \in \mathcal{J}} q_j(v) \phi_j.
\end{equation*}
The operator $P_H$ acts as the inverse of
$(1-\mathcal{C}^*) : V_H \to \widetilde{V}_H$ on $\widetilde{V}_H$. Indeed, for all $v_H\in V_H$,
\begin{equation*}
    P_H(1-\mathcal{C}^*)v_H = P_H v_H = v_H,
\end{equation*}
since $\mathcal{C}^*v_H\in W=\ker P_H$.
We can consequently conclude that $\ker P_H = W$, and thus $(1-P_H)v \in W$ for any $v \in V$. Since we know that $\mathcal{C}^*$ is a projection onto $W$, we have $(1-\mathcal{C}^*)(1-P_H)v=0$. This implies the important property
\begin{equation*}
    (1 - \mathcal{C}^*)v =  (1 - \mathcal{C}^*)P_Hv
\end{equation*}
for any $v\in V$.
This property implies that $V$ and $V_H$ induce the same coarse space. In other words, the functions $\tilde{\phi}_j:=(1-\mathcal{C}^*)\phi_j$ indeed provide a basis of $\widetilde{V}_H$. 

Now, we can introduce the ideal numerical homogenization (with a Petrov-Galerkin formulation)  which reads as follows. Seek $u_H \in V_H$ such that, for all  $v_H\in V_H$, 

\begin{equation}
    a(u_H, (1-\mathcal{C}^*) v_H) = F((1-\mathcal{C}^*) v_H).
\label{eq:discrete}
\end{equation} 
We call this formulation the ideal discrete problem, as it satisfies the following desirable approximation property, namely the preservation of the quantities of interest of the exact solution.

\begin{lemma}
The discrete solution to \eqref{eq:discrete} preserves the quantities of interest of the exact solution to \eqref{e:ProblemWeak}; that is,
\begin{equation*}
    q_j(u) = q_j(u_H).
\end{equation*}
\label{lemma:preservation}
\end{lemma}
\begin{proof}
Due to the conformity $\widetilde{V}_H \subset V$, and denoting $\tilde{v}_H=(1-\mathcal{C}^*) v_H$, we have a Galerkin orthogonality of the form
\begin{equation*}
    a(u-u_H, \tilde{v}_H) = F(\tilde{v}_H) - F(\tilde{v}_H) = 0.
\end{equation*}
By the characterization of $\widetilde{V}_H$ from \cref{thm:decomposition}, namely $a(W, \widetilde{V}_H)=0$ and $V =\widetilde{V}_H \oplus W $, we conclude $u-u_H \in W$. This directly yields $q_j(u-u_H) = 0$ for all $j \in \mathcal{J}$. 
\end{proof}

\subsection{Localized correctors}
In this section, we introduce a crucial feature of the corrector. It is well known that the ideal corrector has an exponential decay property \cite{Altmann21}. The exponential decay of the corrector $\mathcal{C}^*$ motivates the introduction of localized versions of this corrector,
$\mathcal{C}^*_\ell:V \to W$, which only spread information locally. We call $\widehat{\psi}_j = \mathcal{C}^* \phi_j \in W$. Due to~\eqref{eq:Db}, for any function $v_H \in V_H$, it holds
\begin{equation*}
    v_H = \sum_{j \in \mathcal{J}} q_j(v_H) \phi_j.
\end{equation*}
Consequently, we can characterize the application of $\mathcal{C}^*$ on a function in $V_H$ by 
\begin{equation}
    \mathcal{C}^* v_H = \sum_{j \in \mathcal{J}} q_j(v_H) \widehat{\psi}_j.
\label{eq:Corrector}
\end{equation}
Hence a canonical choice of $\mathcal{C}^*_\ell$ replaces the $\widehat{\psi}_j$ with local functions $ \widehat{\psi}_j^\ell \in W(\nei^\ell(\omega_j))$ which solve,  for all $w$ in $W(\nei^\ell(\omega_j))$, 
\begin{equation*}
    a(w, \widehat{\psi}_j^\ell) = a(w, \phi_j).
\end{equation*}
This then leads to the local version $\mathcal{C}^*_\ell: V_H \to W$, namely
\begin{equation}
    \mathcal{C}^*_\ell v_H = \sum_{j \in \mathcal{J}} q_j(v_H)  \widehat{\psi}_j^\ell.
\label{eq:localizedCorrector}
\end{equation}
For the localization analysis, we additionally assume that the
fine-scale space satisfies the local approximation property
\begin{equation}
    \|w\|_{L^2(\omega)}
    \lesssim
    H\|w\|_{V,\nei^r(\omega)}
    \qquad
    \text{for all }w\in W
\label{eq:local-W-approximation}
\end{equation}
for every element patch \(\omega\subseteq\Omega\), where
\(r\in\mathbb N\) is independent of \(H\). In particular,
\begin{equation}
    \|w\|_{L^2(\Omega)}
    \lesssim H\|w\|_V
    \qquad
    \text{for all }w\in W.
\label{eq:global-W-approximation}
\end{equation}

\begin{theorem}[Theorem 3.19 \cite{Altmann21}]
Under the approximation property
\eqref{eq:local-W-approximation}, the localized operator
$\mathcal{C}^*_\ell:V_H \to W$ defined in \eqref{eq:localizedCorrector} satisfies for $v_H \in V_H$,
\begin{equation}
    \lVert(\mathcal{C}^* - \mathcal{C}^*_\ell)v_H\rVert_V \lesssim H^{-s}  \exp(-c \ell) \lVert v_H \rVert_V,
\end{equation}
 where $\mathcal{C}^*$ is defined in \eqref{eq:Corrector} and where $s \geq 0$ and the rate $c$ depends on the size of the $\omega_j$, the mesh regularity and all local stability and continuity constants.
\label{thm:Decay}
\end{theorem}

As a direct consequence of the local correctors $\mathcal{C}^*_\ell$, we define the space
\begin{equation*}
    \widetilde{V}_H^\ell := (1 - \mathcal{C}^*_\ell) V_H.
\end{equation*}
The functions $\tilde{\phi}_j^\ell := (1-\mathcal{C}^*_\ell) \phi_j \in V(\nei^\ell(\omega_j))$, $j \in \mathcal{J}$, can again be characterized in the form of a saddle point problem which does not depend on $\phi_j$. First, let us define a subset  $\mathcal{J}_j \subset \mathcal{J}$  as $\mathcal{J}_j := \left\{ k \in \mathcal{J} \text{ such that } \operatorname{supp}(q_k) \cap \nei^\ell(\omega_j) \neq \emptyset \right\}$, and denote $N_j:=\lvert \mathcal{J}_j \rvert$. Also denote by $\sigma_j$ the function that maps the local index inside the patch $\nei^\ell(\omega_j)$ to the global index, i.e.,  $\sigma_j(n) \in \mathcal{J}_j$ for $0\leq n \leq N_j-1$. Then, there exists a Lagrange multiplier $\mu^{(j)}=[\mu_{0}^{(j)}, \dotsc, \mu_{N_j-1}^{(j)}]^\top \in \mathbb{R}^{N_j}$ such that the pair $(\tilde{\phi}_j^\ell,\mu^{(j)})\in V(\nei^\ell(\omega_j)) \times \mathbb{R}^{N_j}$ solves 
\begin{equation}
    \begin{array}{rll}
     \displaystyle  a(v,\tilde{\phi}_j^\ell) + \sum_{n=0}^{N_j-1} \mu_{n}^{(j)} q_{\sigma_j(n)}(v) &=0  & \forall v \in V(\nei^\ell(\omega_j)) \\
       \displaystyle    q_k(\tilde{\phi}_j^\ell) & = \delta_{kj} & \forall k \in \mathcal{J}_j.
    \end{array}
\label{eq:localProb}
\end{equation}
Let $\mathcal{L}_j^*:V(\nei^\ell(\omega_j)) \to V(\nei^\ell(\omega_j))'$ be the linear form associated with the bilinear form $a$ on $V(\nei^\ell(\omega_j))$
by 
\begin{equation*}
     \langle \mathcal{L}_j^*v,w\rangle_{V' \times V} :=  a(w,v) \quad {\rm for \ all} \ v,w \in V(\nei^\ell(\omega_j)).
\end{equation*}
We denote by $\mathcal{L}_j^{-*}:=(\mathcal{L}_j^*)^{-1}$ its inverse. Also, define $\mathcal{Q}_j: V(\nei^\ell(\omega_j))\to \mathbb{R}^{N_j}$ by ${(\mathcal{Q}_j v)_n = q_{\sigma_j(n)}(v)}$, for all $v \in V(\nei^\ell(\omega_j))$, $0\leq n \leq N_j-1$, and its adjoint $\mathcal{Q}_j^*: \mathbb{R}^{N_j} \to V(\nei^\ell(\omega_j))'$ by ${(\mathcal{Q}_j^*\mu^{(j)}) = \sum_{n=0}^{N_j-1} \mu_{n}^{(j)}  q_{\sigma_j(n)}}$.
Then, the first line of \eqref{eq:localProb} gives 
\begin{equation*}
    \mathcal{L}_j^{*} \tilde{\phi}_j^\ell + \mathcal{Q}_j^* \mu^{(j)} = 0, 
\end{equation*}
which implies 
\begin{equation*}
\tilde{\phi}_j^\ell = -\mathcal{L}_j^{-*}\mathcal{Q}_j^{*}\mu^{(j)}.
\end{equation*}
Then, substituting $\mathcal{Q}_j\tilde{\phi}_j^\ell = e_j$, where $e_j \in \mathbb{R}^{N_j}$ is defined as $[e_j]_n = \delta_{j\sigma_j(n)}$, it follows
\begin{equation*}
- \mathcal{Q}_j\mathcal{L}_j^{-*}\mathcal{Q}_j^{*}\mu^{(j)}=e_j. 
\end{equation*}
Now, we get  $\tilde{\phi}_j^\ell$ by 
\begin{equation*}
  \tilde{\phi}_j^\ell = \mathcal{L}_j^{-*}\mathcal{Q}_j^{*}(\mathcal{Q}_j\mathcal{L}_j^{-*}\mathcal{Q}_j^{*})^{-1} e_j.
\end{equation*}
By denoting the operator $\mathfrak{M}^{(j)}:=\mathcal{Q}_j\mathcal{L}_j^{-*}\mathcal{Q}_j^{*}: \mathbb{R}^{N_j} \to \mathbb{R}^{N_j}$, defined by its entries, for $ 0 \leq n, m \leq N_j-1$, $[\mathfrak{M}^{(j)}]_{nm}=q_{\sigma_j(n)}(\mathcal{L}_j^{-*}q_{\sigma_j(m)})$, we get
\begin{equation*}
\mu^{(j)} = -(\mathfrak{M}^{(j)})^{-1} e_j \quad {\rm and}  \quad
   \tilde{\phi}_j^\ell = - \sum_{n=0}^{N_j-1} \mu^{(j)}_n\mathcal{L}_j^{-*}q_{\sigma_j(n)}.
\end{equation*}
We summarize this procedure in \cref{algo:1}.
\begin{algorithm}[htbp]
\caption{Classical computation of the saddle point problem \eqref{eq:localProb}}
\begin{algorithmic}[1]
\Require Patch $\nei^\ell(\omega_j)$ with a finite index set $\mathcal{J}_j$ of size $N_j$.
\Ensure LOD basis function $\tilde{\phi}_j^\ell$.
\ForAll{$k \in \mathcal{J}_j$}
    \State Seek $\uloc_k^{j,\ell}:=\mathcal{L}_j^{-*}q_k \in V(\nei^\ell(\omega_j))$ such that $\forall v\in V(\nei^\ell(\omega_j))$, $\displaystyle a(v, \uloc_k^{j,\ell}) = q_k(v)$.
\EndFor
\State Assemble $\displaystyle [\mathfrak{M}^{(j)}]_{nm} = q_{\sigma_j(n)}(\uloc_{\sigma_j(m)}^{j,\ell})$ for $  0 \leq n, m \leq N_j-1$.
\State Solve $\mathfrak{M}^{(j)} \mu^{(j)}=-e_j$.
\State Compute $\displaystyle \tilde{\phi}_j^\ell = -\sum_{n=0}^{N_j-1} \mu^{(j)}_n \uloc_{\sigma_j(n)}^{j,\ell}$.
\end{algorithmic}
\label{algo:1}
\end{algorithm}

\subsection{Localized numerical homogenization}

Starting from the ideal formulation \eqref{eq:discrete}, we introduce the practical method in two steps. First, the global corrector $\mathcal C^*$ is replaced by its localized approximation $\mathcal C_\ell^*$, so that the corrected test functions can be computed independently on local patches. The resulting localization error decays exponentially with $\ell$, as quantified in~\cref{thm:Decay}. Second, the corrected right-hand side $F((1-\mathcal C_\ell^*)v_H)$ is replaced by the coarse-scale right-hand side $F(v_H)$. This Petrov--Galerkin perturbation introduces an additional consistency error of order $H$, but avoids the use of localized correctors in the assembly of the load vector.

The resulting localized Petrov--Galerkin method, which is used throughout the remainder of this work, reads as follows. Seek $u_H^\ell\in V_H$ such that
\begin{equation}
    a\bigl(u_H^\ell,(1-\mathcal C_\ell^*)v_H\bigr)=F(v_H)
    \qquad\text{for all }v_H\in V_H.
\label{eq:PGLOD}
\end{equation}
The main advantage of \eqref{eq:PGLOD} lies in the assembly of the discrete system. Since only the test functions are corrected, the row associated with a coarse basis function can be assembled using the localized corrector problems on the corresponding patch and the local coarse trial functions. The patch computations can therefore be carried out independently, without exchanging fine-scale solution data between different patches. In contrast, a Galerkin formulation with corrected functions on both the trial and test sides requires pairings of correctors associated with different overlapping patches and hence additional communication and data aggregation across patches. Furthermore, the uncorrected right-hand side in \eqref{eq:PGLOD} allows the load vector to be assembled solely from the standard coarse finite element basis functions.

Under the assumptions of the localization analysis, formulation \eqref{eq:PGLOD} is well posed and uniformly inf--sup stable if $\ell$ is chosen sufficiently large. Depending on the prefactor in the exponential decay estimate of \cref{thm:Decay}, $\ell$ may be chosen independently of $H$ or proportional to $|\log H|$; see, e.g., \cite{Gallist17,Altmann21,Malqvist21}.

\begin{theorem}[Error estimate for the localized Petrov--Galerkin formulation]
Let $u\in V$ denote the exact solution of \eqref{e:ProblemWeak}, and let $u_H^\ell\in V_H$ denote the solution of \eqref{eq:PGLOD}. Assume that $F(v)=(f,v)$ with $f\in L^2(\Omega)$, and let $\ell$ be sufficiently large to ensure uniform inf--sup stability. Then, for every $j\in\mathcal J$,
\begin{equation}
    \bigl|q_j(u)-q_j(u_H^\ell)\bigr|
    \lesssim
    \left(H^{-s'}\exp(-c\ell)+H\right)\|f\|_{L^2(\Omega)}\|q_j\|_{V'},
\label{eq:PGLOD-QoI-error}
\end{equation}
where $s'\geq 0$ and $c$ is the quantity introduced in \cref{thm:Decay}.
\label{thm:PGLOD-error}
\end{theorem}

\begin{proof}
We first separate the localization error from the perturbation of the right-hand side. For this purpose, let $u_H^{\ell,\mathrm c}\in V_H$ denote the auxiliary solution with corrected right-hand side, defined by
\begin{equation}
    a\bigl(u_H^{\ell,\mathrm c},(1-\mathcal C_\ell^*)v_H\bigr)
    =
    F\bigl((1-\mathcal C_\ell^*)v_H\bigr)
    \qquad\text{for all }v_H\in V_H.
\label{eq:corrected-rhs}
\end{equation}
Define $\widetilde u_H^{\ell,\mathrm c}:=(1-\mathcal C_\ell)u_H^{\ell,\mathrm c}$, where $\mathcal C_\ell$ denotes the localized primal corrector. Furthermore, let $u_H^{\ell,\mathrm G}\in V_H$ solve
\begin{equation}
    a\bigl((1-\mathcal C_\ell)u_H^{\ell,\mathrm G},(1-\mathcal C_\ell^*)v_H\bigr)
    =
    F\bigl((1-\mathcal C_\ell^*)v_H\bigr)
    \qquad\text{for all }v_H\in V_H,
\label{eq:localized-Galerkin}
\end{equation}
and set $\widetilde u_H^{\ell,\mathrm G}:=(1-\mathcal C_\ell)u_H^{\ell,\mathrm G}$. The standard localized Galerkin estimate gives
\begin{equation*}
    \bigl|q_j(u-\widetilde u_H^{\ell,\mathrm G})\bigr|
    \lesssim
    H^{-s}\sqrt{N}\exp(-c\ell)\|u\|_V\|q\|_{V'};
%\label{eq:localized-Galerkin-QoI}
\end{equation*}
see \cite[Theorem~3.23]{Altmann21}. Since $\mathcal C_\ell u_H^{\ell,\mathrm c}\in W$, we have
\begin{equation*}
    q_j(u-u_H^{\ell,\mathrm c})
    =
    q_j(u-\widetilde u_H^{\ell,\mathrm c}).
\end{equation*}
Moreover, subtracting \eqref{eq:corrected-rhs} from \eqref{eq:localized-Galerkin}, and using the exponential decay estimates for the primal and adjoint correctors together with uniform stability, yields
\begin{equation*}
    \|\widetilde u_H^{\ell,\mathrm G}-\widetilde u_H^{\ell,\mathrm c}\|_V
    \lesssim
    H^{-s}\sqrt{N}\exp(-c\ell)\|u\|_V.
%\label{eq:localized-formulation-comparison}
\end{equation*}
Consequently, the triangle inequality and the continuity of $q_j$ imply
\begin{equation}
    \bigl|q_j(u)-q_j(u_H^{\ell,\mathrm c})\bigr|
    \lesssim
    H^{-s}\sqrt{N}\exp(-c\ell)\|u\|_V\|q\|_{V'}.
\label{eq:corrected-rhs-QoI-error0}
\end{equation}
Under the assumption of a uniform mesh, we have $\sqrt{N}=H^{-\frac{d}{2}}$. By setting $s'=s+\frac{d}{2}$, and applying the Lax--Milgram lemma, \eqref{eq:corrected-rhs-QoI-error0} reduces to
\begin{equation}
    \bigl|q_j(u)-q_j(u_H^{\ell,\mathrm c})\bigr|
    \lesssim
    H^{-s'}\exp(-c\ell)\|f\|_{L^2(\Omega)}\|q\|_{V'}.
\label{eq:corrected-rhs-QoI-error}
\end{equation}
It remains to estimate the perturbation caused by the uncorrected right-hand side. Setting $e_H^\ell:=u_H^{\ell,\mathrm c}-u_H^\ell$ and subtracting \eqref{eq:PGLOD} from \eqref{eq:corrected-rhs}, we obtain
\begin{equation*}
    a\bigl(e_H^\ell,(1-\mathcal C_\ell^*)v_H\bigr)
    =
    -F(\mathcal C_\ell^*v_H)
    \qquad\text{for all }v_H\in V_H.
%\label{eq:rhs-perturbation-equation}
\end{equation*}
Since $\mathcal C_\ell^*v_H\in W$, the fine-scale approximation property \eqref{eq:global-W-approximation}, the stability of $\mathcal C_\ell^*$, and the uniform inf--sup stability of \eqref{eq:PGLOD} give
\begin{equation*}
    \|u_H^{\ell,\mathrm c}-u_H^\ell\|_V
    \lesssim
    H\|f\|_{L^2(\Omega)}.
%\label{eq:PG-rhs-perturbation}
\end{equation*}
Hence,
\begin{equation*}
    \bigl|q_j(u_H^{\ell,\mathrm c})-q_j(u_H^\ell)\bigr|
    \lesssim
    H\|f\|_{L^2(\Omega)}\|q_j\|_{V'}.
\end{equation*}
Combining this estimate with \eqref{eq:corrected-rhs-QoI-error} proves the assertion.
\end{proof}

\subsection{Offline--online strategy and hybrid approach}
In this section, we turn to the offline assembly of the stiffness matrix associated with problem \eqref{eq:PGLOD}. We then show how this assembly can be embedded into a hybrid framework in which classical solutions of the local problems are replaced by quantum-computing-based procedures.

First, recalling that  $V_H = {\rm span}\{\phi_j, \ j \in \mathcal{J} \}$, the stiffness matrix $\mathcal{A}$ is defined by its entries, for $i,j\in \mathcal{J}$, using the notation and the computation of $\tilde{\phi}_j^\ell$ introduced in \cref{algo:1},
\begin{equation*}
\mathcal{A}_{ji} = a(\phi_i, \tilde{\phi}_j^\ell) = -\sum_{n=0}^{N_j-1} \mu_n^{(j)} a(\phi_i, \uloc_{\sigma_j(n)}^{j, \ell})= -\sum_{n=0}^{N_j-1}\mu_n^{(j)} \s_{i\sigma_j(n)}^{(j)}, 
\end{equation*}
where $\s_{ik}^{(j)}:=a(\phi_i,  \uloc_{k}^{j, \ell})$ for $i,k\in \mathcal{J}_j$.  The classical offline assembly of the stiffness matrix is summarized in \cref{algo:2}.

\begin{algorithm}[htbp]
\caption{Classical Offline Assembly of the stiffness matrix}
\begin{algorithmic}[1]

\Require Patch $\nei^\ell(\omega_j)$ with a finite index set $\mathcal{J}_j$ of size $N_j$.
\Ensure LOD stiffness matrix $\mathcal{A}$. \Comment{Offline assembly of the stiffness matrix.}
\ForAll{$j \in \mathcal{J}$}
\ForAll{$k \in \mathcal{J}_j$}
    \State Seek $\uloc_k^{j,\ell}:=\mathcal{L}_j^{-*}q_k \in V(\nei^\ell(\omega_j))$ such that $\forall v\in V(\nei^\ell(\omega_j))$, $\displaystyle a(v, \uloc_k^{j,\ell}) = q_k(v)$.
\EndFor
\State Assemble $\displaystyle [\mathfrak{M}^{(j)}]_{nm} = q_{\sigma_j(n)}(\uloc_{\sigma_j(m)}^{j,\ell})$ for $  0 \leq n, m \leq N_j-1$.
\State Solve $\mathfrak{M}^{(j)} \mu^{(j)}=-e_j$.

\ForAll{$i \in \mathcal{J}_j$}
\State $  \displaystyle \mathcal{A}_{ji} = -\sum_{n=0}^{N_j-1} \mu_n^{(j)} a(\phi_i, \uloc_{\sigma_j(n)}^{j,\ell})= -\sum_{n=0}^{N_j-1}\mu_n^{(j)} \s_{i\sigma_j(n)}^ {(j)}.$
\EndFor
\EndFor
\end{algorithmic}
\label{algo:2}
\end{algorithm}
Once the stiffness matrix has been assembled in the offline phase, problem \eqref{eq:PGLOD} can be solved in the online phase without explicit knowledge of the LOD, or corrected, basis functions. Indeed, their contribution is already encoded in the entries of $\mathcal{A}$. Consequently, these basis functions need not be stored globally after the offline phase. Moreover, since the LOD basis functions are independent of the right-hand side, the stiffness matrix can be reused for different choices of $f$. Therefore, for each new right-hand side, only the online stage needs to be performed, consisting of assembling the corresponding coarse load vector and solving the coarse linear system.

In particular, in this work, we choose $f\in L^2(\Omega)$, and consider $F(v_H) = (f,v_H)$. The online computation is summarized in \cref{algo:3}.

\begin{algorithm}[htbp]
\caption{Classical Online Computation}
\begin{algorithmic}[1]
\Require Stiffness matrix $\mathcal{A}$; function $f$.
\Ensure Solution vector $U$. 
\ForAll{$j \in \mathcal{J}$} \Comment{Online assembly of the right-hand side $\mathbf{F}$.}
\State $\mathbf{F}_j =(f,\phi_j)_{\omega_j}.$
\EndFor
\State Solve $\mathcal{A} U = \mathbf{F}$.
\end{algorithmic}
\label{algo:3}
\end{algorithm}

\begin{remark}
Writing the coarse Petrov--Galerkin solution as
$u_H^\ell=\sum_{j\in\mathcal J}U_j\phi_j$,
the Lagrange property $q_i(\phi_j)=\delta_{ij}$ of the coarse trial basis gives $U_j=q_j(u_H^\ell)$ for all $j\in\mathcal J$. Thus, the solution vector computed in the online stage directly represents the quantities of interest of the coarse trial solution. 
\label{rem:ComputationQ}
\end{remark}

From \cref{algo:2} and \cref{algo:3}, we observe that the local problems enter the computation only during the offline phase, and only through scalar quantities derived from their solutions. More precisely, the assembly of the local matrices $\mathfrak{M}^{(j)}$ and of the global stiffness matrix $\mathcal{A}$ requires quantities of the form $q_i(\uloc_k^{j,\ell})$ and $a(\phi_i,\uloc_k^{j,\ell})$. Thus, the full local solutions $\uloc_k^{j,\ell}$ do not have to be stored or explicitly reconstructed, provided that these scalar quantities can be evaluated.

This structure is particularly well suited to a hybrid quantum--classical strategy. In the proposed approach, the local problems are solved by a quantum algorithm, while the quantities required for the classical assembly are obtained by quantum measurements. Since quantum algorithms naturally provide information through measurements of observables, the role of the quantum device is not to output the full fine-scale functions $\uloc_k^{j,\ell}$, but rather to provide the scalar data needed for the offline assembly.

For each patch $\nei^\ell(\omega_j)$, these scalar quantities have to be evaluated for all relevant pairs of local indices and for the two types of quantities above. Hence, the number of distinct scalar quantities required for the assembly is of order $2N_j^2$. The actual number of quantum measurements, or shots, needed to estimate these quantities depends on the prescribed measurement tolerance. As discussed below, $N_j$ is relatively small; in typical situations, $N_j=\mathcal{O}(2^d(\ell+1)^d)$, with $\ell$ chosen of order $|\log H|$ in the theoretical analysis, while in practical computations values between $2$ and $4$ are typically sufficient. Importantly, $N_j$ depends only on the coarse-patch size and not on the number of fine degrees of freedom used to discretize the local problem. This is crucial, since the fine-scale discretization on a patch may involve on the order of $\mathcal{O}(h^{-d})$ degrees of freedom. Therefore, the number of scalar measurements needed for the hybrid assembly is much smaller than the dimension of the fine-scale local problem. The resulting hybrid quantum--classical procedure is summarized in \cref{algo:4}.

\begin{algorithm}[htbp]
\caption{Hybrid Framework for the assembly of the stiffness matrix}
\begin{algorithmic}[1]

\Require Patch $\nei^\ell(\omega_j)$ with a finite index $\mathcal{J}_j$ of size $N_j$.
\Ensure LOD stiffness matrix $\mathcal{A}$.  \Comment{Offline computation}
\ForAll{$j \in \mathcal{J}$} \\
\State{\textbf{Quantum computation}:} \\
\ForAll{$i, k \in \mathcal{J}_j$}  
\State $\widehat{\s}_{ik}^{(j)} \gets \text{measure }\mathfrak{a}_{ik}^{(j)} \coloneq a(\phi_i, \uloc_k^{j,\ell})$.
\State $\widehat{\mathfrak{m}}_{ik}^{(j)} \gets \text{measure }\mathfrak{m}_{ik}^{(j)} \coloneq q_i(\uloc_k^{j,\ell})$.
\EndFor \\
\State{\textbf{Classical computation}:} \\
\State Assemble $\displaystyle [\widehat{\mathfrak{M}}^{(j)}]_{nm} =  \widehat{\mathfrak{m}}_{\sigma_j(n)\sigma_j(m)}^{(j)}$ for $  0 \leq n, m \leq N_j-1$.
\State Solve $\widehat{\mathfrak{M}}^{(j)} \widehat{\mu}^{(j)}=-e_j$.

\ForAll{$i \in \mathcal{J}_j$}
\State $ \displaystyle \widehat{\mathcal{A}}_{ji} = -\sum_{n=0}^{N_j-1}\widehat{\mu}_n^{(j)} \widehat{\s}_{i\sigma_j(n)}^{(j)}.$

\EndFor
\EndFor
\end{algorithmic}
\label{algo:4}
\end{algorithm}

Afterwards, the online computation of the global problem continues to follow \cref{algo:3}. In the following section, we illustrate this hybrid method for a specific choice of quantities of interest and describe the corresponding quantum computing procedures of \cref{algo:4}.

\section{Discretization and quantum implementation}
\label{sec.Quantum}

In this section, we present the discretization and quantum implementation of the proposed methodology for a specific choice of quantities of interest. Also, from now, for technical reasons, we consider only the diffusion part of \eqref{eq:Problem}, i.e., $\mathfrak{b}=0$ and $\mathfrak{c}=0$, and assume $\mathfrak{A} \in L^\infty(\Omega, \mathbb{R}^{d \times d}_{\rm sym})$. The question of quantum preconditioning of the full problem \eqref{eq:Problem} is still open. Besides this theoretical question, the methodology developed in what follows can also be applied to the full problem. In this context, the bounds $\alpha$ and $\beta$ defined in \eqref{eq:abound}, are just $\theta$ and $\Theta$ respectively, and from now, we track their dependence in the various estimates. 

\subsection{Coarse scale space}
First, we assume that $\mathcal{G}_H$ is a regular partition of $\Omega$ made of squares in two dimensions and cubes in three dimensions (see \cite[Remark 6.4]{Deiml25}). In this configuration, we redefine the mesh-size $H$ to be the side-length of the squares or cubes (\cref{thm:PGLOD-error} still holds with this definition of $H$).
\medskip

Let $Q_1(\G_H)$ denote the space of piecewise multilinear polynomials, meaning that each variable of the polynomials appears with power at most $1$. The space of globally continuous multilinear polynomials reads
\begin{equation*}
\cS^1(\G_H):= C^0(\Omega)\cap Q_1(\G_H).
\end{equation*}
In this section, we specify a concrete coarse space $V_H$, defined as the standard $Q_1$ finite element space, i.e., 
\begin{equation*}
V_H:=\cS^1(\G_H) \cap V.
\end{equation*}
The set of free vertices of $\G_H$ (the degrees of freedom) is denoted by
\begin{equation*}
      \N_H:=\{j\in\overline\Omega \text{ such that } 
           j\text{ is a vertex of }\G_H\text{ and }j\notin\partial\Omega\}.
\end{equation*}
From this point onward, for each vertex $j \in \mathcal{N}_H$, let $\phi_j \in V_H$ denote the associated nodal basis function, uniquely determined by the nodal values 
\begin{equation*}
    \phi_j(j) = 1 \ {\rm and }  \  \phi_j(i) = 0 \ {\rm for \ all} \  i \neq j \in \mathcal{N}_H.
\end{equation*}
These nodal basis functions form a basis of $V_H$, i.e., $V_H = {\rm span}\{\phi_j, \ j \in \mathcal{N}_H$\}. With this choice of nodal basis functions, we have $\omega_j = \supp \phi_j = \cup \{T \in \mathcal{G}_H \ {\rm such \ that} \ j \ {\rm is \ a \ vertex \ of } \ T \}$.  We also introduce $\N_H^{(j)}$ the set of local vertices of $\nei^\ell(\omega_j)$. To align with the notations introduced in the previous sections, we set $\mathcal{J}= \mathcal{N}_H$, where $\mathcal{N}_H$ denotes the set of free vertices of $\G_H$. For each $j \in \mathcal{N}_H$, we let $\mathcal{J}_j = \N_H^{(j)}$. Therefore, in this setting, that $N = \lvert \mathcal{N}_H \rvert$ and $ N_j  = \lvert\N_H^{(j)}\rvert$.

A natural choice of quantities of interest is
\begin{equation}
    q_j(v):=\sum_{i\in\mathcal N_H}(M^{-1})_{ij}(\phi_i,v)
    \qquad\text{for }j\in\mathcal N_H,
\label{eq:QoIHat0}
\end{equation}
where $M_{ij}:=(\phi_i,\phi_j)$. These functionals are the coordinate functionals of the $L^2$-orthogonal projection $P_H:V\to V_H$ defined by
\begin{equation*}
    P_Hv:=\sum_{j\in\mathcal N_H}q_j(v)\phi_j.
\end{equation*}
In particular, given the Lagrange properties $q_j(\phi_k)=\delta_{jk}$, \cref{rem:ComputationQ} holds. Indeed, if the solution of \eqref{eq:PGLOD} is written as
\begin{equation*}
    u_H^\ell=\sum_{j\in\mathcal N_H}U_j\phi_j,
\end{equation*}
then
\begin{equation}
    U_j=q_j(u_H^\ell),
\label{eq:coefficient-qoi}
\end{equation}
so that the coefficient vector directly represents the approximations of the quantities of interest appearing in \cref{thm:PGLOD-error}. The functionals $q_j$ are global and therefore do not satisfy the locality assumptions of the abstract localization result \cref{thm:Decay} directly, but they are exponentially quasi-local owing to the decay of the inverse mass matrix on quasi-uniform meshes; see, e.g.,~\cite{Demko77}. For ease of the practical quantum implementation, we follow the original LOD construction of~\cite{Malqvist14} for local patch computations and employ the (local) weighted Cl\'ement-type averaging functionals, defined as 
\begin{equation}
    \qq_j(v):=\frac{(\phi_j,v)_{\nei^\ell(\omega_j)}}{(\phi_j,1)_{\nei^\ell(\omega_j)}}
    \qquad\text{for }j\in\mathcal N_H,
\label{eq:QoIHat}
\end{equation}
and the associated quasi-interpolation operator
\begin{equation*}
    I_Hv:=\sum_{j\in\mathcal N_H}\qq_j(v)\phi_j.
\end{equation*}
That is, all occurrences of generic functionals $q_j$ in local problems \eqref{eq:localProb} and in \cref{algo:1,algo:2,algo:4} are instantiated by local functionals $\qq_j$. 

It should be noted that this practical technique does not change the method. In fact, although $I_H$ is not a projection, it has the same kernel as $P_H$. Indeed,
\begin{equation}
    \ker(P_H)
    =
    \left\{v\in V \ {\rm such \ that} \ (\phi_j,v)=0\text{ for all }j\in\mathcal N_H\right\}
    =
    \ker(I_H).
\label{eq:same-kernel-PH-IH}
\end{equation}
Consequently, both families define the same ideal fine-scale space $W$ and hence the same ideal multiscale test space $\widetilde V_H$. The local approximation and stability properties of $I_H$, including the local right-inverse property required for the construction of $\qq$-Lagrange representatives, are established in \cite{Malqvist14}. Hence, the implemented Cl\'ement-type basis approximates the same ideal test space exponentially, and the corresponding error is covered by the localization term in~\cref{thm:PGLOD-error}. Since the nodal trial basis remains unchanged, the coefficient identity \eqref{eq:coefficient-qoi} remains valid (still with the quantities of interest \eqref{eq:QoIHat0}), although $\qq_j(\phi_k)\neq\delta_{jk}$.

In addition, for the chosen partition $\mathcal{G}_H$, each subdomain $\omega_j$  
is a square (or a cube) domain with side length of $2H$. The corresponding $\ell$th-order patch  $\nei^\ell(\omega_j)$, used in the localization and decay analysis, is consequently a square (or a cube) domain with side length of $2(\ell+1)H$, containing $2^d(\ell+1)^d$ coarse elements.

\begin{remark}
For nodes $j \in \mathcal{N}_H$ near the boundary $\partial \Omega$, the patch $\nei^\ell(\omega_j)$ may not be fully contained in the domain, i.e., 
$\nei^\ell(\omega_j) \not \subset \Omega$. In this case, for ease of implementation, we use the technique introduced in \cite[Section 3.4]{ThesisMohr}, which is based on domain extension and allows handling of homogeneous Dirichlet boundary conditions. The consideration of such cases is explained in~\cref{sec:DomainExtension}. 
\end{remark}
Within this setting, we first recall some useful estimates for what follows. First, from standard elliptic theory, there exist constants $\theta_j, \Theta_j >0$ such that, for all $v \in V(\nei^\ell(\omega_j))$, 
\begin{equation}
\theta \lVert v \rVert_{V,\nei^\ell(\omega_j)}^2  \leq \theta_j \lVert v \rVert_{V,\nei^\ell(\omega_j)}^2 \leq \langle  \mathcal{L}^{*}_jv,v\rangle_{V' \times V} \leq \Theta_j \lVert v \rVert_{V,\nei^\ell(\omega_j)}^2 \leq \Theta \lVert v \rVert_{V,\nei^\ell(\omega_j)}^2.
\label{eq:Ljbound}
\end{equation}
Also, it holds, for $j\in \mathcal{N}_H$,
\begin{equation}
    (\phi_j, 1)_{\nei^\ell(\omega_j)}\approx H^d, \quad \lVert \phi_j \rVert_{V, \nei^\ell(\omega_j)} \approx H^{\frac{d}{2}-1}, \quad \lVert \phi_j \rVert_{L^2(\nei^\ell(\omega_j))} \approx H^{\frac{d}{2}}.
\label{eq:EstimateHat}
\end{equation}
Then, by the Lax–Milgram lemma, we obtain, for all $j\in \N_H$ and $k \in \N_H^{(j)}$, 
\begin{equation*}
    \lVert \mathcal{L}^{-*}_j\qq_k\rVert_{V,\nei^\ell(\omega_j)} \leq \frac{1}{\theta_j} \lVert \qq_k \rVert_{V',\nei^\ell(\omega_j)},
\end{equation*}
where
\begin{equation*}
    \lVert \qq_k \rVert_{V',\nei^\ell(\omega_j)} = \sup_{v \in V(\nei^\ell(\omega_j)),v\neq 0} \frac{(\phi_k, v)_{\nei^\ell(\omega_j)}}{\lVert v \rVert_{V,\nei^\ell(\omega_j)}}\frac{1}{(\phi_k, 1)_{\nei^\ell(\omega_j)}}.
\end{equation*}
Now, applying the Cauchy-Schwarz and then the Poincaré inequality on the quantity $(\phi_k, v)_{\nei^\ell(\omega_j)}$, and using \eqref{eq:EstimateHat}, it follows 
\begin{equation}
    \lVert \qq_k\rVert_{V',\nei^\ell(\omega_j)} \lesssim \frac{H^{\frac{d}{2}} 2(\ell+1)H}{H^d} \lesssim 2(\ell+1)H^{1-\frac{d}{2}}, 
\label{eq:estimateqk}
\end{equation}
and then
\begin{equation}
    \lVert \mathcal{L}^{-*}_j\qq_k\rVert_{V,\nei^\ell(\omega_j)}  \lesssim \frac{2(\ell+1)}{\theta_j}   H^{1-\frac{d}{2}}.
\label{eq:EstimateLq}
\end{equation}

In what follows,  we define for $y \in \mathbb{R}^n$ the Euclidean norm $\lVert y \rVert_2= \sqrt{y^\top y}$. Then, the corresponding matrix norm is given for $A \in \mathbb{R}^{n \times m}$ by $\lVert A \rVert_2= \sqrt{\lambda_{\rm max}(A^\top A)}$ where $\lambda_{\rm max}$ denotes the largest eigenvalue. By $A^+$ we denote the Moore--Penrose inverse, which is equivalent to $A^{-1}$ for invertible matrices, and by $\kappa(A) \coloneqq \|A\|_2 \|A^+\|_2$ its condition number.

\subsection{Quantum local solver}
\label{sec.QuantumSolver}
To solve the local problems, a fine discretization of the patch $\nei^\ell(\omega_j)$ is required. Let $L_{\max} \in \mathbb{N}$ be the fine discretization level. We then consider the sequence of fine grids $\mathcal{G}_{h_L}^{(j)}$ of size $h_L = 2^{-L}H$ for $L = 1, \dots, L_{\max}$, obtained by sub-dividing each coarse cell of the patch $2^{L}$ times in each spatial direction, and we denote $h \coloneqq h_{L_{\max}}$. We define, for $j \in \mathcal{N}_H$ and $1 \le L \le L_{\max}$, the finite element space $V_{h_L}^{(j)} \coloneqq \cS^1(\mathcal{G}_{h_L}^{(j)}) \cap V(\nei^\ell(\omega_j))$. We further denote $Q_h \coloneqq Q_1(\mathcal{G}_h)$ to be the space of (discontinuous) piecewise multilinear polynomial functions.

We consider the local problems of finding the approximation $\uloc_{h, k}^{j,\ell}$ to $\uloc_k^{j, \ell}$ defined in \cref{algo:2}. Specifically, the local problems read as follows. Seek $\uloc_{h,k}^{j,\ell} \in V_{h}^{(j)}$ such that, for all $v \in V_{h}^{(j)}$, 
\begin{equation*}
    a(v, \uloc_{h,k}^{j,\ell}) = (\mathfrak{A}  \nabla v, \nabla \uloc_{h,k}^{j,\ell}) = \qq_k(v).
\end{equation*}
From now, let us consider a single patch $\nei^\ell(\omega_j)$. For simplicity we do not track the dependence to $j$ and $\ell$ in the objects defined in what follows. 
Following \cite{Deiml25}, we decompose the bilinear form~$a$ into the multiplication with the coefficient $\mathfrak{A}$, defined as a function acting on the vector-valued space $Q_h^d$, and the gradient restricted to $V_h^{(j)}$. The image of this restriction is then also contained in $Q_h^d$. This space admits a $L^2$ orthonormal basis $\{\zeta_0, \dots, \zeta_{2^d|\mathcal{G}_h| - 1}\}$, such that each function has support on only one cell of $\mathcal{G}_h$. Written in this basis, the multiplication with the coefficient~$\mathfrak{A}$ then becomes a block diagonal matrix that we denote $D_\mathfrak{A}$. Further, due to the orthonormality, the $L^2$ inner product in the space $Q_h$, and by extension $Q_h^d$, is equivalent to the Euclidean inner product. 

For each inner vertex $z$ of each fine grid $\mathcal{G}_{h_L}$, we denote the corresponding nodal basis function~$\lambda_z^{(L)}$. One could consider the standard basis of $V_h^{(j)}$ consisting of the functions $\lambda_z^{(L_{\max})}$, however, this would lead to a linear system whose condition number depends on the fine-scale $h$. Instead, we consider a special generating system of $V_h^{(j)}$, related to the BPX preconditioner~\cite{Bramble89,Osw94}, namely
\begin{equation*}
\mathcal{Z} := \{2^{-L(2 - d)/2} \lambda_{z}^{(L)}, \ L = 1, \dots, L_{\max}, \; z \text{ is an inner vertex of } \mathcal{G}_{h_L} \} = \{\xi_0, \dots, \xi_{|\mathcal{Z}| - 1}\}.
\end{equation*}
We assemble the gradient into a matrix $G$ using this generating system, meaning, for $0 \leq s \leq |\mathcal{Z}| - 1$,
\[
\nabla \xi_s = \sum_{t = 0}^{2^d|\mathcal{G}_h| - 1} G_{ts} \zeta_t.
\]
Then, we can write the bilinear form $a$ using this matrix, for $0 \leq s,t \leq |\mathcal{Z}| - 1$, as
\[
a(\xi_s, \xi_t) = (G^\top D_\mathfrak{A} G)_{st}.
\]
Crucially, this triple product has a bounded condition number independent of $h$.

First, following \Cref{algo:4}, we need to measure the quantities $\mathfrak{m}_{ik}^{(j)} \coloneq \qq_i(\uloc_{h,k}^{j,\ell})$. Let $\mathbf{\qq}_k \in \mathbb{R}^{|\mathcal{Z} |}$ be the vector whose entries are the values of the quantity of interest $\qq_k$ evaluated at the elements of $\mathcal{Z}$, i.e.,\ $[\mathbf{\qq}_k]_s = \qq_k(\xi_s)$. Then, 
\begin{equation} \label{eq:quantum-approximation-a}
\mathfrak{m}_{ik}^{(j)} \coloneq \qq_i(\uloc_{h,k}^{j,\ell}) = \mathbf{\qq}_i^\top (G^\top D_\mathfrak{A} G)^+ \mathbf{\qq}_k.
\end{equation}
Next, we need to measure the quantities $\s_{ik}^{(j)} \coloneq a(\phi_i, \uloc_{h,k}^{j,\ell})$.
However, for $i\in \N_H^{(j)}$, $\phi_i$ may correspond to a vertex on the boundary of the patch, in which case it is not in the space $V_h^{(j)}$. Consequently, the gradient is not correctly calculated as $G \,\mathbf{\qq}_{i}$. Instead, we directly assemble the vector representation of $\nabla \phi_i \in Q_h^d$, denoted $\mathbf{g}_i$. This is possible since $\mathbf{g}_i$ can be efficiently computed analytically. We then compute
\begin{equation} \label{eq:quantum-approximation-b}
\s_{ik}^{(j)} = a(\phi_i, \uloc_{h,k}^{j,\ell}) = \mathbf{g}_{i}^\top D_{\mathfrak{A}} G(G^\top D_{\mathfrak{A}} G)^+ \mathbf{\qq}_{k} = \mathbf{g}_{i}^\top D_{\mathfrak{A}}^{ \frac{1}{2}} (G^\top D_{\mathfrak{A}}^{ \frac{1}{2}})^+ \mathbf{\qq}_{k}.
\end{equation}
Note, that \eqref{eq:quantum-approximation-b} may not be simplified further since $(G D_\mathfrak{A}^{ \frac{1}{2}})^+ = (D_\mathfrak{A}^{ \frac{1}{2}})^+ G^+$ does not hold in general.

To perform this computation on a quantum computer, we must construct \emph{block encodings} of $G$ and $D_\mathfrak{A}$. Block encodings are a concept specific to quantum computing. The full definition is not needed here. For further details see \cite[Section~4]{Gilyen19} and \cite[Section~4]{Deiml25}. We recall the following.
\begin{definition}[Block encoding]
    Let $n, m \in \mathbb N$ and $X \in \mathbb{R}^{n \times m}$. A \emph{block encoding} of $X$ is a quantum algorithm that encodes $X$. Its complexity is characterized by its normalization $\gamma(X) \ge \|X\|_2$ and runtime $T(X) \in \mathbb{N}$. A block encoding of a vector $x \in \mathbb{R}^n$ is simply a block encoding of $x$ as a $n \times 1$ matrix. Similarly a block encoding of a scalar $\lambda \in \mathbb{R}$ is a block encoding of $\lambda$ as a $1 \times 1$ matrix.
\end{definition}
Most importantly, block encodings allow linear algebraic computations with a logarithmic or polylogarithmic number of operations in terms of the dimensions $n$ and $m$, i.e.,
\[\mathcal{O}(\polylog nm) \coloneq \mathcal{O}((\log nm)^\nu)\]
for some fixed $\nu > 0$. To preserve this computational advantage, we cannot extract a full matrix or vector from the quantum computer, and are instead limited to scalar measurements. The process of measuring is known as amplitude estimation in this context \cite{Brassard02}.

\begin{lemma}[Amplitude estimation] \label{lem:amplitude-estimation}
    Let $\lambda \in \mathbb{R}$, which is given as a block encoding. Then, for any failure probability $\delta > 0$ and absolute error tolerance $\tol > 0$ there is a quantum algorithm $\mathcal E(\lambda, \delta, \tol)$ that estimates $\lambda$ with the given probability and error in time
    \[
    \mathcal{O}(\tol^{-1}\gamma(\lambda) \log \delta^{-1} T(\lambda)).
    \]
    In other words,
    \[
    \Pr(|\mathcal{E}(\lambda, \delta, \tol) - \lambda| > \tol) < \delta.
    \]%
\label{lemma:Amplitude_estimation}%
\end{lemma}%

In \cite{Deiml25}, we presented a construction of almost optimal block encodings for $G$ and $D_\mathfrak{A}$. Our main results are summarized in the following theorem.

\begin{theorem}[Quantum Realization of FEM] \label{thm:quantum-fem}
Assume that $\mathfrak{A}$ is classically computable. Then there exists block encodings of $G$ and $D_\mathfrak{A}$ with runtime and normalization, 
\[
T(G), T(D_\mathfrak{A}) \in \mathcal{O}(\polylog h^{-1}), \qquad \gamma(G) \lesssim \sqrt{\log h^{-1}}\|G\|_2, \quad\text{and}\quad \gamma(D_\mathfrak{A}) \le \Theta.
\]
Further $\kappa(G) \in \mathcal{O}(1)$ and $\kappa(D_\mathfrak{A}) \le \Theta / \theta$.
Block encodings of $\mathbf{\qq}_{i}$, respectively $\mathbf{g}_i$,   can be constructed with runtime and normalization
\[
T(\mathbf{\qq}_i) \in \mathcal{O}(\polylog h^{-1}), \qquad \gamma(\mathbf{\qq}_i) = \|\mathbf{\qq}_i\|_2,
\]
and analogously for $\mathbf{g}_i$.
\end{theorem}

The construction of block encodings for vectors is given in \cite[Appendix~A]{Deiml25}. Note, that the concrete right-hand sides given here can easily be encoded in polylogarithmic runtime, but in general this is a more difficult problem. The bounds on the runtime and normalization indicate that the forward evaluation of the bilinear form is efficiently possible. Solving the corresponding linear system requires a quantum linear system solver, which can be characterized as follows.

\begin{lemma}[Pseudoinverse of block encodings] \label{lem:inversion}
Let $n, m \in \mathbb{N}$ and $X \in \mathbb{R}^{n \times m}$. Assume we have access to a block encoding of $X$, as well as a lower bound $\sigma_{-} > 0$ of its smallest non-zero singular value, i.e.,\ \(\sigma_{-} \le \sigma_{\min}(X) = \|X^+\|_2^{-1}\). Then for any $\tol > 0$ we can construct a block encoding~$Y$ approximating the pseudoinverse $X^+$, such that
\[
\|Y - X^+\|_2 \le \tol, \qquad \gamma(Y) = \sigma_{-}^{-1}, \qquad T(Y) \in \mathcal{O}(T(X) \gamma(X)/\sigma_{-} \log \tol^{-1}).
\]
Note that $\gamma(X) / \sigma_{-} \ge \kappa(X) = \sigma_{\max}(X) / \sigma_{\min}(X)$.
\end{lemma}

See \cite[Theorem~41]{Gilyen19} as an example of a quantum linear system solver with this asymptotic complexity. More efficient solvers like \cite{Low24,An22} achieve an even better scaling with the condition number.
Finally, we give the following characterization of multiplications of block-encoded matrices or vectors.

\begin{lemma}[Multiplication of block encodings] \label{lem:multiplication}
    Let $n, m, p \in \mathbb{N}$ and $X \in \mathbb{R}^{n \times m}, Y \in \mathbb{R}^{m \times p}$. Assume we have access to block encodings of $X$ and $Y$. Then, we can construct a block encoding of the product $XY$, such that
    \[
    \gamma(XY) = \gamma(X)\gamma(Y), \qquad T(XY) = T(X) + T(Y).
    \]
\label{lemma:Mul}
\end{lemma}

Note that this implies, in particular, that one has to be quite careful to multiply matrices with large condition number on a quantum computer, as this can lead to a large ratio between the normalization~$\gamma(XY)$ and norm~$\|XY\|_2$. This in turn would cause the encoded values to be hard to measure according to \Cref{lem:amplitude-estimation}. The decomposition given by~\eqref{eq:quantum-approximation-a} and \eqref{eq:quantum-approximation-b}, however, is chosen to avoid such issues.

\begin{theorem}[Quantum solution of patch problems]
    Let $\tol, \delta > 0$ and $i, k \in \mathcal{N}_H^{(j)}$. There exists a quantum algorithm that computes an approximation of $\widehat{\mathfrak{m}}_{ik}^{(j)} \approx \qq_i(\uloc_{h,k}^{j,\ell})$ or $\widehat{\s}_{ik}^{(j)} \approx a(\phi_i, \uloc_{h,k}^{j,\ell})$ such that the probability of an absolute error larger than $\tol$ is at most $\delta$. The algorithm has runtime
\begin{equation*}
    \mathcal{O}\left((\ell + 1)^2 \frac{\Theta}{\theta^2} H^{2-d}\tol^{-1} \polylog(h^{-1}) \log (\tol^{-1}) \log (\delta^{-1}) \right) \quad {\rm for} \ \widehat{\mathfrak{m}}_{ik}^{(j)}
\end{equation*}
and
\begin{equation*}
\mathcal{O}\left((\ell + 1) \frac{\Theta}{\theta} \tol^{-1} \polylog(h^{-1}) \log (\tol^{-1}) \log (\delta^{-1}) \right) \quad {\rm for} \ \widehat{\mathfrak{a}}_{ik}^{(j)}.
\end{equation*}
\label{eq:ErrorMeas}
\end{theorem}
\begin{proof}
    The construction follows \eqref{eq:quantum-approximation-a} and \eqref{eq:quantum-approximation-b} by using \Cref{thm:quantum-fem,lem:multiplication,lem:inversion}. Using \Cref{lem:amplitude-estimation}, the complexity is given by
    \begin{equation} \label{eq:complexity}
    \mathcal{O} \left(\gamma\tol^{-1} \polylog (h^{-1}) \log (\tol^{-1}) \log (\delta^{-1}) \right),
    \end{equation}
where $\gamma$ is the normalization of the measured block encoding, which has to be estimated for both measured quantities.  Let us start with $\qq_i(\uloc_{h,k}^{j, \ell})$, which is computed as $\mathbf{\qq}_i^\top (G^\top D_\mathfrak{A} G)^+ \mathbf{\qq}_k$. With~\cref{lemma:Mul}, it follows that
\begin{equation}
    \gamma(\mathbf{\qq}_i^\top (G^\top D_\mathfrak{A} G)^+ \mathbf{\qq}_k) = \gamma(\mathbf{\qq}_i)\gamma( (G^\top D_\mathfrak{A} G)^+)  \gamma(\mathbf{\qq}_k).
\label{eq:gammaM}
\end{equation}
Now, we propose to estimate the two different values involved in the right-hand side of \eqref{eq:gammaM}. First, using \Cref{thm:quantum-fem,lem:inversion}, the normalization of the pseudoinverse is given by
\begin{equation}
    \gamma((G^\top D_\mathfrak{A}G)^+) \lesssim \|(D_\mathfrak{A}^{ \frac{1}{2}}G)^+\|_2^2 = \frac{\kappa(D_\mathfrak{A}^{ \frac{1}{2}} G)^2}{\|D_\mathfrak{A}^{ \frac{1}{2}}G\|_2^2} \lesssim \frac\Theta\theta\|D_\mathfrak{A}^{ \frac{1}{2}}G\|_2^{-2} \lesssim \frac{\Theta}{\theta^2} \|G\|_2^{-2}.
\label{eq:gammaGD}
\end{equation}
To estimate the Euclidean norm of $\mathbf{\qq}_k$, observe that for any vector $\mathbf{c} \in \mathbb{R}^{|\mathcal{Z}|}$ we have
    \[
    \left\|\sum_{s = 0}^{|\mathcal{Z}|-1} \mathbf{c}_s \xi_s\right\|_{V, \nei^\ell(\omega_j)}^2 = \mathbf{c}^\top G^\top G\mathbf{c} = \|G\mathbf{c}\|_2^2 \qquad\text{and}\qquad \qq_k\left(\sum_{s = 0}^{|\mathcal{Z}|-1} \mathbf{c}_s \xi_s\right) = \mathbf{\qq}_k^\top \mathbf{c}.
    \]
    By setting $\mathbf{c} = \mathbf{\qq}_k$, we can thus obtain the following bound,
    \begin{align*}
    \|\mathbf{\qq}_k\|_2^2 &= \mathbf{\qq}_k^\top \mathbf{\qq}_{k} = \qq_k\left(\sum_{s = 0}^{|\mathcal{Z}|-1} [\mathbf{\qq}_k]_s \xi_s\right) \le \|\qq_k\|_{V',\nei^\ell(\omega_j)} \left\|\sum_{s = 0}^{|\mathcal{Z}|-1} [\mathbf{\qq}_{k}]_s \xi_s\right\|_{V, \nei^\ell(\omega_j)} \\
    & \leq \|\qq_k\|_{V',\nei^\ell(\omega_j)} \|G\mathbf{\qq}_{k}\|_2 \le \|\qq_k\|_{V',\nei^\ell(\omega_j)} \|G\|_2 \|\mathbf{\qq}_{k}\|_2.
    \end{align*}
Dividing by $\|\mathbf{\qq}_{k}\|_2$ on both sides, and using estimate \eqref{eq:estimateqk}, we get
\begin{equation}
        \gamma(\mathbf{\qq}_{k}) = \|\mathbf{\qq}_{k}\|_2 \le \|\qq_k\|_{V',\nei^\ell(\omega_j)} \|G\|_2 = 2(\ell + 1)H^{1-\tfrac{d}{2}} \|G\|_2.
    \label{eq:gammaqk}
\end{equation}
At the end, combining \eqref{eq:gammaGD} and \eqref{eq:gammaqk}, leads to
\begin{equation} \label{eq:gammaM-final}
    \gamma(\qq_i(\uloc_{h,k}^{j,\ell})) = \gamma\big(\mathbf{\qq}_i^\top (G^\top D_\mathfrak{A} G)^+ \mathbf{\qq}_{k}\big) \lesssim \frac{\Theta}{\theta^2}(\ell + 1)^2 H^{2 - d}.
\end{equation}

Now, we turn to the normalization for $a(\phi_i, \uloc_{h,k}^{j,\ell})$, which is approximated as $\mathbf{g}_{i}^\top D_{\mathfrak{A}}^{ \frac{1}{2}} (G^\top D_{\mathfrak{A}}^{ \frac{1}{2}})^+ \mathbf{\qq}_{k}$. Using the same procedure as previously, it remains to estimate $\gamma(\mathbf{g}_{i})$. Using estimate \eqref{eq:EstimateHat}, it follows
\begin{equation}
\gamma(\mathbf{g}_{i}) = \|\phi_i\|_{V, \nei^\ell(\omega_i)} \approx H^{\tfrac{d}{2}-1}.
\label{eq:gammag}
\end{equation}
At the end, combining  \eqref{eq:gammaGD}, \eqref{eq:gammaqk} and \eqref{eq:gammag}, it follows
\begin{equation} \label{eq:gammaA-final}
\gamma(a(\phi_i, \uloc_{h,k}^{j,\ell})) = \gamma\big(\mathbf{g}_{i}^\top D_\mathfrak{A}^{\frac{1}{2}} (G^\top D_\mathfrak{A}^{\frac{1}{2}})^+ \mathbf{\qq}_{k}\big) \lesssim \frac{\Theta}{\theta}(\ell + 1). 
\end{equation}
The statement of the theorem then results from combining \eqref{eq:complexity} with \eqref{eq:gammaM-final} or \eqref{eq:gammaA-final} respectively.
\end{proof}

\section{Sensitivity analysis}
\label{sec.Sens}

In the proposed hybrid approach, we consider the stiffness matrix $\widehat{\mathcal{A}}$ given by \cref{algo:4}, obtained from measurements of quantum processes. It is well known that such measurements are inherently noisy \cite{Nielsen_Chuang_2010}, both due to the stochastic noise of the idealized model as described in~\cref{lem:amplitude-estimation} and imperfections of the hardware realization. Consequently, $\widehat{\mathcal{A}}$ is an approximation of the true stiffness matrix $\mathcal{A}$ from \cref{algo:2}. In what follows, we discuss the impact of measurement noise and assess the quality of this approximation for both the stiffness matrix and the solution of the global problem. Some reminders about sensitivity of linear problems are given in \cref{sec:Reminder}. 

Keeping the notations of \cref{algo:2} and \cref{algo:4}, we denote the perturbed measurements by $\widehat{\mathfrak{M}}^{(j)}:=\mathfrak{M}^{(j)}+ \mathcal{N}_{\mathfrak{M}}^{(j)}$ and  $\widehat{\s}_{ik}^{(j)} := \s_{ik}^{(j)} +\mathcal{N}_{\s_{ik}}^{(j)}$, where $\mathcal{N}$ denotes the noise, and let us define  \begin{equation}
     \mathcal{E}_{\s} = \max_{j \in \N_H} \max_{ i,k \in \N_H^{(j)}}   \lvert \mathcal{N}^{(j)}_{\s_{ik}} \rvert, \; {\rm and} \; \mathcal{E}_\mathfrak{M} = \max_{j \in \N_H} \lVert \mathcal{N}^{(j)}_{\mathfrak{M}} \rVert_{\max}.
\label{eq:maxNoise}
    \end{equation}
Recall that, for each $j \in \N_H$, $\widehat{\mu}^{(j)}$ solves 
\begin{equation}
    \widehat{\mathfrak{M}}^{(j)}\widehat{\mu}^{(j)}=-e_j.
\label{eq:Mmu}
\end{equation}
To account for the perturbations, we replace \eqref{eq:Mmu}, by the following perturbed system 
 \begin{equation}
    (\mathfrak{M}^{(j)}+\mathcal{N}^{(j)}_{\mathfrak{M}})(\mu^{(j)}+\delta \mu^{(j)})=-e_j,
\label{eq:MmuPerturbed}
\end{equation}
where $\delta \mu^{(j)}$ represents the change in the solution.

\begin{lemma}
 The perturbed stiffness matrix $\widehat{\mathcal{A}}$ from \cref{algo:4} satisfies 
\begin{equation*}
     \widehat{\mathcal{A}} = \mathcal{A} + \delta \mathcal{A}, 
\end{equation*}
where $\mathcal{A}$ is the true stiffness matrix from \cref{algo:2}, and where $\delta \mathcal{A}$ is defined by its entries for $i,j \in \N_H$, by
\begin{equation*}
    \delta  \mathcal{A}_{ji} \approx \sum_{n=0}^{N_j-1} (\mu_n^{(j)}\mathcal{N}_{\s_{i\sigma_j(n)}}^{(j)} + \delta \mu_n^{(j)} \s^{(j)}_{i\sigma_j(n)})
\end{equation*}
up to higher order terms. In particular, it holds
\begin{equation*}
    \lVert  \delta \mathcal{A} \rVert_2 \lesssim  H^{-\frac{d}{2}}  (2\ell+3)^d \Theta \left( \mathcal{E}_\s + \frac{\Theta^2}{\theta} 2(\ell+1)(2 \ell+3)^d\mathcal{E}_\mathfrak{M} \right),
\end{equation*} 
where $\mathcal{E}_\s$ and $\mathcal{E}_\mathfrak{M}$ are defined in \eqref{eq:maxNoise}.
\label{lemma:sensitivity}
\end{lemma}

\begin{proof}
First, the perturbed stiffness matrix $\widehat{\mathcal{A}}$ is defined by its entries for $i,j \in \N_H$ as 
\begin{equation*}
\begin{aligned}
   \widehat{\mathcal{A}}_{ji} &= - \sum_{n=0}^{N_j-1} \widehat{\mu}_n^{(j)} \widehat{\s}^{(j)}_{i\sigma_j(n)}\\
   &=   - \sum_{n=0}^{N_j-1} (\mu_n^{(j)}  + \delta \mu^{(j)}_n)(\s^{(j)}_{i\sigma_j(n)} +\mathcal{N}_{\s_{i\sigma_j(n)}}^{(j)}) \\
   &= -  \sum_{n=0}^{N_j-1} (\mu^{(j)}_n \s_{i\sigma_j(n)}^{(j)} + \mu^{(j)}_n \mathcal{N}_{\s_{i\sigma_j(n)}}^{(j)} + \delta \mu^{(j)}_n  \s^{(j)}_{i\sigma_j(n)}) \\
       &= \mathcal{A}_{ji} + \delta  \mathcal{A}_{ji},
\end{aligned}
\end{equation*}
ignoring the second order terms $\delta \mu_n^{(j)} \mathcal{N}_{\mathfrak{a}_{i\sigma_j(n)}}^{(j)}$ for the noise and with
\begin{equation*}
    \delta \mathcal{A}_{ji} = - \sum_{n=0}^{N_j-1} (\mu^{(j)}_n \mathcal{N}_{\s_{i\sigma_j(n)}}^{(j)} + \delta \mu^{(j)}_n \s^{(j)}_{i\sigma_j(n)}).
\end{equation*}
Now, we propose to get a bound for $\delta \mathcal{A}_{ji}$. First, using the triangle inequality, and then, the Cauchy--Schwarz inequality, it follows
\begin{equation}
\lvert \delta \mathcal{A}_{ji}\rvert \leq \lVert\mu^{(j)} \rVert_2 \Big(\sum_{n=0}^{N_j-1} \lvert \mathcal{N}_{\s_{i\sigma_j(n)}}^{(j)} \rvert^2\Big)^{\frac{1}{2}} + \lVert \delta \mu^{(j)} \rVert_2 \Big(\sum_{n=0}^{N_j-1} \lvert \s^{(j)}_{i\sigma_j(n)} \rvert^2\Big)^{\frac{1}{2}}.
\label{eq:in1}
\end{equation}
Next, we bound each term appearing in the right-hand side of \eqref{eq:in1}. 

First, we bound $\|\delta \mu^{(j)}\|_2$. It turns out with our choice of quantities of interest \eqref{eq:QoIHat}, that $\mathfrak{M}^{(j)}$ is given by the composition of the normalized mass matrix with $\mathcal{L}^{-*}_j$, and then, using \eqref{eq:Ljbound}, it follows
\begin{equation*}
    \lVert (\mathfrak{M}^{(j)})^{-1} \rVert_2  \lesssim  \Theta_j.
\end{equation*}
Considering the perturbed system \eqref{eq:MmuPerturbed}, assuming that the measurement noise is sufficiently small, it follows with \eqref{eq:firstorder} that 
\begin{equation*}
    \frac{\lVert \delta \mu^{(j)} \rVert_2 }{\lVert\mu^{(j)}\rVert_2} \approx \kappa \big(\mathfrak{M}^{(j)}\big) \frac{\lVert \mathcal{N}^{(j)}_{\mathfrak{M}} \rVert_2}{\lVert \mathfrak{M}^{(j)}\rVert_2}.
\end{equation*}
Since $\mu^{(j)}=-(\mathfrak{M}^{(j)})^{-1} e_j$, we have 
\begin{equation}
     \lVert\mu^{(j)} \rVert_2 \leq  \lVert \big(\mathfrak{M}^{(j)}\big)^{-1} \rVert_2   \lVert e_j \rVert_2 \leq \lVert \big(\mathfrak{M}^{(j)}\big)^{-1} \rVert_2 \leq  \Theta_j,
    \label{eq:ineqa}
\end{equation}
and then
\begin{equation}
  \lVert \delta \mu^{(j)} \rVert_2 \leq   \Theta_j^2  \lVert \mathcal{N}^{(j)}_{\mathfrak{M}}\rVert_2,
 \label{eq:ineqb}
\end{equation}
given that $\kappa (\mathfrak{M}^{(j)})=\lVert \mathfrak{M}^{(j)}\rVert_2\lVert (\mathfrak{M}^{(j)})^{-1}\rVert_2$. Additionally,
\begin{equation}
    \lVert \mathcal{N}^{(j)}_{\mathfrak{M}}\rVert_2 \leq (2\ell+3)^d \mathcal{E}_\mathfrak{M}.
 \label{eq:ineqc} 
\end{equation}
where $(2\ell+3)^d$ corresponds to the number of nodes in $\nei^\ell(\omega_j)$. 
Turning to the last term of~\eqref{eq:in1}, we have, using estimates \eqref{eq:EstimateHat} and \eqref{eq:EstimateLq}, 
\begin{align*}
     \lvert a(\phi_i, \mathcal{L}^{-*}_j\qq_k) \rvert & \leq   \Theta \lVert \phi_i \rVert_{V,\nei^\ell(\omega_j)} \ \lVert \mathcal{L}^{-*}_j\qq_k \rVert_{V,\nei^\ell(\omega_j)} \\
   & \lesssim \Theta  H^{\frac{d}{2}-1} \left( \frac{2 (\ell+1)}{\theta_j}H^{1-\frac{d}{2}} \right)\\
& \lesssim  \frac{ 2(\ell+1) \Theta}{\theta_j},
\end{align*}
which leads to 
\begin{equation}
\begin{aligned}
\displaystyle
   \Big( \sum_{n=0}^{N_j-1} \lvert \s^{(j)}_{i\sigma_j(n)} \rvert^2 \Big)^{\frac{1}{2}}&= \Big( \sum_{k\in \N_H^{(j)}} \lvert a(\phi_i, \mathcal{L}^{-*}_j\qq_k) \rvert^2 \Big)^{\frac{1}{2}}\\
   & \lesssim \Big( \sum_{k \in\N_H^{(j)}} \left(\frac{ 2(\ell+1) \Theta}{\theta_j}\right)^2 \Big)^{\frac{1}{2}}\\   
   & \lesssim (2\ell+3)^{\frac{d}{2}} \left(\frac{ 2(\ell+1) \Theta}{\theta_j}\right).
\end{aligned}
 \label{eq:ineqd}
\end{equation}
Finally, the remaining term of \eqref{eq:in1} is simply bounded by
\begin{equation}
   \Big(\sum_{n=0}^{N_j-1} \lvert \mathcal{N}_{\s_{i\sigma_j(n)}}^{(j)} \rvert^2\Big)^{\frac{1}{2}} \leq (2\ell+3)^{\frac{d}{2}} \mathcal{E}_\s. 
 \label{eq:ineqe}
\end{equation}
Gathering the bounds \eqref{eq:ineqa}, \eqref{eq:ineqb}, \eqref{eq:ineqc}, \eqref{eq:ineqd} and \eqref{eq:ineqe}, it follows
\begin{equation*}
    \lvert \delta \mathcal{A}_{ji}\rvert \lesssim \Theta_j(2\ell+3)^{\frac{d}{2}}  \mathcal{E}_\s + \Theta_j^2  \frac{\Theta}{\theta_{j}} 2(\ell+1) (2\ell+3)^{\frac{d}{2}} (2 \ell+3)^d\mathcal{E}_\mathfrak{M}. 
\end{equation*}
At the end, taking the bounds of the $\theta_j$ and $\Theta_j$, it follows
\begin{equation*}
    \lVert  \delta \mathcal{A} \rVert_2 \lesssim (2\ell+3)^{\frac{d}{2}}H^{-\frac{d}{2}}  \left(\Theta (2\ell+3)^{\frac{d}{2}} \mathcal{E}_\s +   \frac{\Theta^3}{\theta} 2(\ell+1)  (2 \ell+3)^{\frac{3d}{2}}\mathcal{E}_\mathfrak{M} \right),
\end{equation*}
since the number of rows of the stiffness matrix $\mathcal{A}$, and thus $\delta\mathcal{A}$, is about $H^{-d}$, and each row has at most $(2\ell+3)^d$ non-zero entries.
\end{proof}

\begin{remark}
We note that the bound obtained bound on $\delta \mathcal{A}$ is pessimistic. Under the assumption that the noise terms are independent and normally distributed with zero mean and variance $\sigma^2$, the quantities $\mathcal{E}_\s$ and $\mathcal{E}_\mathfrak{M}$ may be replaced by their characteristic scale $\sigma$.
\end{remark}

If $\delta \mathcal{A}$ is small enough so that $\mathcal{A} + \delta \mathcal{A}$ remains invertible, then using \eqref{eq:firstorder} and noting that the LOD stiffness matrix has a condition number $\kappa(\mathcal{A}) \approx H^{-2}$, we obtain
\begin{equation}
\frac{\lVert \delta u_H \rVert_2}{\lVert u_H \rVert_2}  \approx H^{-2} \frac{\lVert \delta \mathcal{A}\rVert_2}{\lVert\mathcal{A}\rVert_2}.
\label{eq:ErrCond}
\end{equation}
In particular, \Cref{eq:ErrorMeas} and \cref{lemma:sensitivity} guarantee that $\delta \mathcal{A}$ can be made arbitrarily small.
This estimate suggests that, as the coarse mesh size decreases, the influence of quantum noise becomes more pronounced. As a result, achieving the same level of accuracy on the macro-scale solution necessitates a higher measurement precision.

\section{Numerical Results and Discussions}
\label{sec.num}
In this section, we evaluate the proposed hybrid procedure on two-dimensional problems.  We recall we are considering the following problem. 
Let $\Omega = [0,1]^2$. We seek $u$ such that
\begin{equation*}
    - {\rm div}(\mathfrak{A} \nabla u) = f \quad {\rm in} \ \Omega, \quad u =0  \ {\rm on} \ \partial \Omega.
\end{equation*}

\subsection{Settings}
We implement a classical simulation of the quantum solver described in \cref{sec.QuantumSolver}, using Python. For this simulation, the normalization~$\gamma$ of each measured quantity is computed exactly according to \Cref{thm:quantum-fem,lem:inversion,lem:multiplication}, along with the encoded scalar. Subsequently, Monte-Carlo amplitude estimation is simulated, which approximates the measured quantities as an average of $N_\text{samples} \in \mathbb{N}$ binomial samples. The accuracy of this approximation scales like $\tol \approx \gamma (N_\text{samples})^{-\frac{1}{2}}$. We choose this Monte-Carlo sampling, since the resulting random variables are easy to characterize, but recall from \Cref{lem:amplitude-estimation} that there exist estimation methods with better asymptotic behavior. 
\medskip

We denote by $u_H^\ell$ the approximated solutions.  
For the sake of comparison, we compute reference solutions, denoted $u_{\rm ref}$, using usual $Q_1$ finite element method on a fine discretization of~$\Omega$ in squares of size $h=2^{-10}$. Then, we compare the quantities of interest $q_i(u_H^\ell)$ with $q_i(u_{\rm ref})$ for each interior nodes of the coarse mesh, where we recall that $q_i$ is defined as \eqref{eq:QoIHat0}, by computing the following relative error, 
\begin{equation}
    \mathcal{E}_\mathcal{Q}= \frac{\max_{i\in \mathcal{N}_H}\lvert q_i(u_{\rm ref})-q_i(u_H^\ell)\rvert}{\max_{i\in \mathcal{N}_H} \lvert q_i(u_{\rm ref}) \rvert} = \frac{\max_{i\in \mathcal{N}_H}\lvert q_i(u_{\rm ref})-U_i\rvert}{\max_{i\in \mathcal{N}_H} \lvert q_i(u_{\rm ref}) \rvert},
\label{eq:err}
\end{equation}
where $U_i$ is the $i$th entry of the solution vector.

\subsection{Test cases}

We consider two test cases.  In the first test case, we consider a periodic diffusion coefficient given by 
\begin{equation*}
    \mathfrak{A}_1(x,y)=2\left(2+\sin\left(2\pi \frac{x}{\varepsilon}\right)\sin\left(2\pi \frac{y}{\varepsilon}\right)\right)^{-1} \mathbf{I}_2,
\end{equation*}
with $\varepsilon=2^{-8}$, and \(\mathbf{I}_2\) denotes the \(2\times2\) identity matrix. A representation of this diffusion coefficient is given in \cref{fig:A_test_1}. The right-hand side is chosen as $f(x,y)=2\pi^2\sin(\pi x)\sin(\pi y)$.  In the second test case, we consider a non-periodic diffusion coefficient given by 
\begin{equation*}
    \mathfrak{A}_2(x,y)= \left(1+10^{-8} + \frac{1}{2} \sin \left( \lfloor x+y \rfloor +  \left \lfloor \frac{x}{\varepsilon} \right\rfloor + \left \lfloor \frac{y}{\varepsilon} \right \rfloor \right) 
    + \frac{1}{2} \cos \left( \lfloor x-y \rfloor +  \left \lfloor \frac{x}{\varepsilon} \right \rfloor + \left \lfloor \frac{y}{\varepsilon} \right \rfloor \right) \right)\mathbf{I}_2,
\end{equation*}
with $\varepsilon=2^{-7}$.  A representation of this diffusion coefficient is given in \cref{fig:A_test_2}. The right-hand side is chosen as  $f(x,y)=1$. For Test Case 1, we compute the LOD approximation for coarse mesh sizes $H \in \{2^{-3},2^{-4}, 2^{-5}, 2^{-6},2^{-7} \}$, ensuring that $H \gg \varepsilon$, and for $\ell \in \{ 2,3,4 \}$.  For Test Case 2, we consider $H \in \{2^{-3},2^{-4}, 2^{-5}, 2^{-6} \}$, with the same values of $\ell$, again maintaining $H \gg \varepsilon$. In both cases, configurations for which the patches $\nei^\ell(\omega_j)$ is larger than $\Omega$ are discarded (i.e., for which $2(\ell+1)H>1$). For each considered case, we use a local fine mesh with size  $h=2^{-10}$, which is sufficiently small to capture the oscillations of the diffusion coefficient $\mathfrak{A}_1$ or~$\mathfrak{A}_2$.

\begin{figure}[htbp]
    \centering

    \begin{subfigure}{0.45\linewidth}
        \centering
        \includegraphics[width=0.8\linewidth]{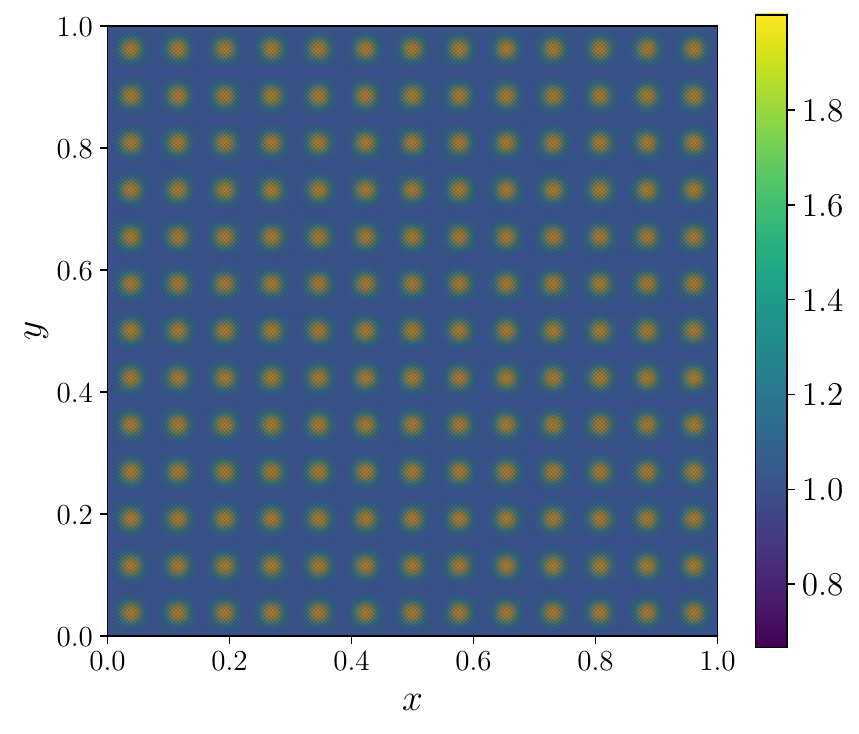}
        \caption{Diffusion coefficient $\mathfrak{A}_1$ of Test Case 1.}
        \label{fig:A_test_1}
    \end{subfigure}
    \hfill
    \begin{subfigure}{0.45\linewidth}
        \centering
        \includegraphics[width=0.8\linewidth]{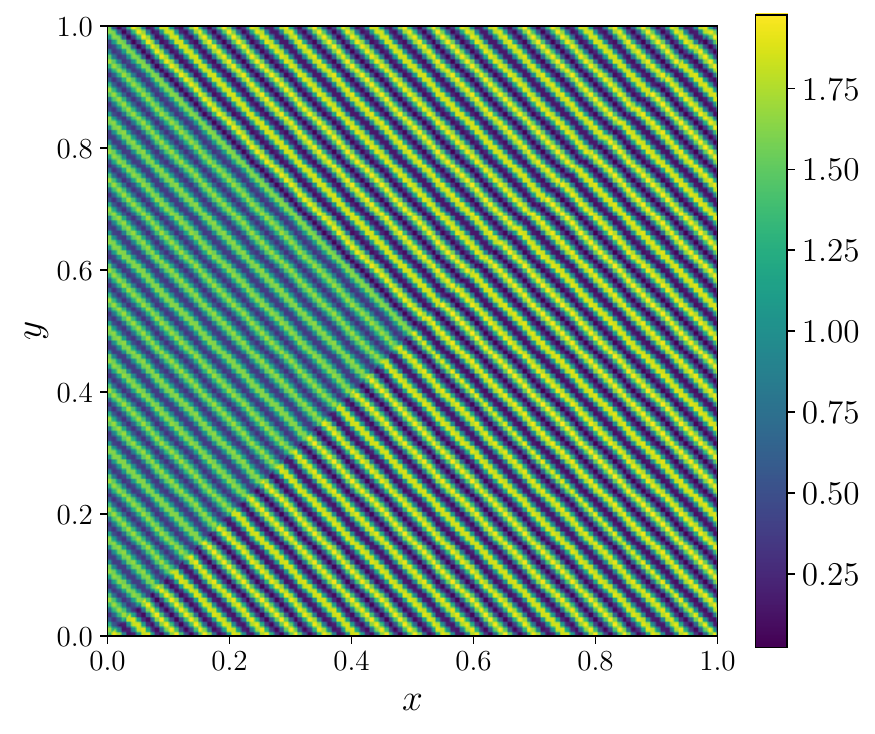}
        \caption{Diffusion coefficient $\mathfrak{A}_2$ of Test Case 2.}
        \label{fig:A_test_2}
    \end{subfigure}

    \caption{Diffusion coefficients for the two test cases.}
    \label{fig:coefficients}
\end{figure}

\subsection{Results}

We recall that, within our approach, the objective is to approximate certain quantities of interest. In \cref{fig:Error1,fig:Error2}, we compare the approximations obtained with our method to the corresponding quantities of interest extracted from reference solutions.

\begin{figure}[htpb]
\centering	
\begin{tikzpicture}

\begin{loglogaxis}[
  name = plot1, 
   at={(0.\textwidth,0.\textwidth)},
    width=0.3\textwidth,
    xmin=0.0078125, xmax=0.125,
    ymin=6e-3, ymax=1,
    xlabel={Coarse element size $H$},
    ylabel={Relative Error $\mathcal{E}_\mathcal{Q}$},
    ylabel near ticks,
    yticklabel pos=left,
    tick align=outside,
    tick pos=left,
    grid=both,
    minor y tick num=9,
    minor grid style={dotted, black!30},
    major grid style={line width=0.5pt, draw=black!50, dashed},
    legend style={at={(1.05,1)}, anchor=north west, fill=none},
    enlargelimits=true,
    xtick={0.0078125,0.015625,0.03125,0.0625,0.125},
    xticklabels={$2^{-7}$,$2^{-6}$,$2^{-5}$,$2^{-4}$,$2^{-3}$},
    ylabel style={yshift=-5pt}, % move closer to y-axis ticks
    xlabel style={yshift=-5pt}  % move farther from x-axis ticks
]

%-----------------------------
% Deterministic FE REV
%-----------------------------
\addplot[
    color=unia-pink,
    line width=1.5pt,
    mark=square,
    mark size=2pt,
    every mark/.append style={solid},
] table[x=H, y=err_mean]{dat_files_Test1/err_l2_samples0.dat};

%-----------------------------
% Stochastic delta_1
%-----------------------------
\addplot[
    color=unia-orange,
    line width=1.5pt,
    mark=square,
    mark size=2pt,
    every mark/.append style={solid, fill=gray},
    error bars/.cd,
        y dir=both,
        y explicit
] table[
    x=H,
    y=err_mean,
    y error=err_std
]{dat_files_Test1/err_l2_samples_1e8.dat};

%-----------------------------
% Stochastic delta_2 with variable ±std error bars
%-----------------------------
\addplot[
    color=unia-lightblue,
    line width=1.5pt,
    mark=square,
    mark size=2pt,
    every mark/.append style={solid, fill=gray},
    error bars/.cd,
        y dir=both,
        y explicit
] table[
    x=H,
    y=err_mean,
    y error=err_std
]{dat_files_Test1/err_l2_samples_1e9.dat};

%-----------------------------
% Legend
%-----------------------------
\end{loglogaxis}

\node at (plot1.south) [yshift=-45pt] {(a) $\ell = 2$};

\begin{loglogaxis}[
    name=plot2,
    at={(0.33\textwidth,0.\textwidth)},
    width=0.3\textwidth,
    xmin=0.0078125, xmax=0.125,
    ymin=6e-3, ymax=1,
    xlabel={Coarse element size $H$},
    ylabel={Relative Error $\mathcal{E}_\mathcal{Q}$},
    ylabel near ticks,
    yticklabel pos=left,
    tick align=outside,
    tick pos=left,
    grid=both,
    minor y tick num=9,
    minor grid style={dotted, black!30},
    major grid style={line width=0.5pt, draw=black!50, dashed},
    legend style={at={(1.05,1)}, anchor=north west, fill=none},
    enlargelimits=true,
    xtick={0.0078125,0.015625,0.03125,0.0625,0.125},
    xticklabels={$2^{-7}$,$2^{-6}$,$2^{-5}$,$2^{-4}$,$2^{-3}$},
     ylabel style={yshift=-5pt}, % move closer to y-axis ticks
    xlabel style={yshift=-5pt}  % move farther from x-axis ticks
]

%-----------------------------
% Deterministic FE REV
%-----------------------------
\addplot[
    color=unia-pink,
    line width=1.5pt,
    mark=square,
    mark size=2pt,
    every mark/.append style={solid},
] table[x=H, y=err_mean]{dat_files_Test1/err_l3_samples0.dat};

%-----------------------------
% Stochastic delta_1
%-----------------------------
\addplot[
    color=unia-orange,
    line width=1.5pt,
    mark=square,
    mark size=2pt,
    every mark/.append style={solid, fill=gray},
    error bars/.cd,
        y dir=both,
        y explicit
] table[
    x=H,
    y=err_mean,
    y error=err_std
]{dat_files_Test1/err_l3_samples_1e8.dat};

%-----------------------------
% Stochastic delta_2 with variable ±std error bars
%-----------------------------
\addplot[
    color=unia-lightblue,,
    line width=1.5pt,
    mark=square,
    mark size=2pt,
    every mark/.append style={solid, fill=gray},
    error bars/.cd,
        y dir=both,
        y explicit
] table[
    x=H,
    y=err_mean,
    y error=err_std
]{dat_files_Test1/err_l3_samples_1e9.dat};

\end{loglogaxis}

\node at (plot2.south) [yshift=-45pt] {(b) $\ell = 3$};

\begin{loglogaxis}[
    name = plot3, 
    width=0.3\textwidth,
    at={(0.66\textwidth,0\textwidth)},
    xmin=0.0078125, xmax=0.0625,
    ymin=6e-3, ymax=1,
    xlabel={Coarse element size $H$},
    ylabel={Relative Error $\mathcal{E}_\mathcal{Q}$},
    ylabel near ticks,
    yticklabel pos=left,
    tick align=outside,
    tick pos=left,
    grid=both,
    minor y tick num=9,
    minor grid style={dotted, black!30},
    major grid style={line width=0.5pt, draw=black!50, dashed},
    legend style={at={(-2.2,1.8)}, anchor=north west, fill=none},
    legend cell align=left,
    legend columns=1,
    enlargelimits=true,
    xtick={0.0078125,0.015625,0.03125,0.0625},
    xticklabels={$2^{-7}$,$2^{-6}$,$2^{-5}$,$2^{-4}$},
    ylabel style={yshift=-5pt}, % move closer to y-axis ticks
    xlabel style={yshift=-5pt}  % move farther from x-axis ticks
]
%-----------------------------
% Deterministic FE REV
%-----------------------------
\addplot[
    color=unia-pink,
    line width=1.5pt,
    mark=square,
    mark size=2pt,
    every mark/.append style={solid},
] table[x=H, y=err_mean]{dat_files_Test1/err_l4_samples0.dat};
%-----------------------------
% Stochastic delta_1
%-----------------------------
\addplot[
    color=unia-orange,
    line width=1.5pt,
    mark=square,
    mark size=2pt,
    every mark/.append style={solid, fill=gray},
    error bars/.cd,
        y dir=both,
        y explicit
] table[
    x=H,
    y=err_mean,
    y error=err_std
]{dat_files_Test1/err_l4_samples_1e8.dat};

%-----------------------------
% Stochastic delta_2 with variable ±std error bars
%-----------------------------
\addplot[
    color=unia-lightblue,
    line width=1.5pt,
    mark=square,
    mark size=2pt,
    every mark/.append style={solid, fill=gray},
    error bars/.cd,
        y dir=both,
        y explicit
] table[
    x=H,
    y=err_mean,
    y error=err_std
]{dat_files_Test1/err_l4_samples_1e9.dat};

%-----------------------------
% Legend
%-----------------------------
\legend{Classical Implementation, Quantum Impl. $N_\text{samples} = 10^{8}$, Quantum Impl. $N_\text{samples} = 10^{9}$}
\end{loglogaxis}

\node at (plot3.south) [yshift=-45pt] {(c) $\ell = 4$};

\end{tikzpicture}
\caption{Relative errors \eqref{eq:err} for Test Case 1 between the LOD approximations and a reference solution  for $\ell=2,3,4$. Results are shown for both the classical LOD and its quantum variant with varying numbers of samples. The classical implementation follows \cref{algo:2}, the quantum implementation follows \cref{algo:4}. For the quantum variant, results are averaged over six realizations of the full procedure, and the corresponding error bars are shown.}
\label{fig:Error1}
\end{figure}
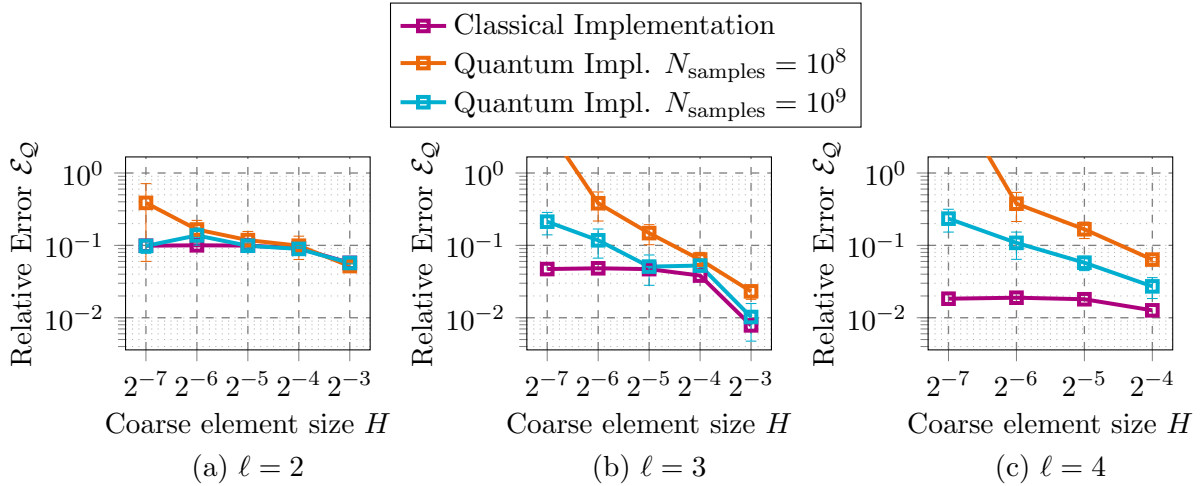

\begin{figure}[htpb]
\centering	
\begin{tikzpicture}
\begin{loglogaxis}[
    name = plot1, 
   at={(0.\textwidth,0.\textwidth)},
    width=0.3\textwidth,
    xmin=0.015625, xmax=0.125,
    ymin=5e-3, ymax=1,
    xlabel={Coarse element size $H$},
    ylabel={Relative Error $\mathcal{E}_\mathcal{Q}$},
    ylabel near ticks,
    yticklabel pos=left,
    tick align=outside,
    tick pos=left,
    grid=both,
    minor y tick num=9,
    minor grid style={dotted, black!30},
    major grid style={line width=0.5pt, draw=black!50, dashed},
    legend style={at={(1.05,1)}, anchor=north west, fill=none},
    enlargelimits=true,
    xtick={0.015625,0.03125,0.0625,0.125},
    xticklabels={$2^{-6}$,$2^{-5}$,$2^{-4}$,$2^{-3}$},
    ylabel style={yshift=-5pt}, % move closer to y-axis ticks
    xlabel style={yshift=-5pt}  % move farther from x-axis ticks
]

%-----------------------------
% Deterministic FE REV
%-----------------------------
\addplot[
    color=unia-pink,
    line width=1.5pt,
    mark=square,
    mark size=2pt,
    every mark/.append style={solid},
] table[x=H, y=err_mean]{dat_files_Test2/err_l2_samples0.dat};

%-----------------------------
% Stochastic delta_1
%-----------------------------
\addplot[
    color=unia-orange,
    line width=1.5pt,
    mark=square,
    mark size=2pt,
    every mark/.append style={solid, fill=gray},
    error bars/.cd,
        y dir=both,
        y explicit
] table[
    x=H,
    y=err_mean,
    y error=err_std
]{dat_files_Test2/err_l2_samples_1e8.dat};

%-----------------------------
% Stochastic delta_2 with variable ±std error bars
%-----------------------------
\addplot[
    color=unia-lightblue,,
    line width=1.5pt,
    mark=square,
    mark size=2pt,
    every mark/.append style={solid, fill=gray},
    error bars/.cd,
        y dir=both,
        y explicit
] table[
    x=H,
    y=err_mean,
    y error=err_std
]{dat_files_Test2/err_l2_samples_1e9.dat};

\end{loglogaxis}

\node at (plot1.south) [yshift=-45pt] {(a) $\ell = 2$};

\begin{loglogaxis}[
      name=plot2,
    at={(0.33\textwidth,0.\textwidth)},
    width=0.3\textwidth,
    xmin=0.015625, xmax=0.125,
    ymin=5e-3, ymax=1,
    xlabel={Coarse element size $H$},
    ylabel={Relative Error $\mathcal{E}_\mathcal{Q}$},
    ylabel near ticks,
    yticklabel pos=left,
    tick align=outside,
    tick pos=left,
    grid=both,
    minor y tick num=9,
    minor grid style={dotted, black!30},
    major grid style={line width=0.5pt, draw=black!50, dashed},
    legend style={at={(1.05,1)}, anchor=north west, fill=none},
    enlargelimits=true,
    xtick={0.015625,0.03125,0.0625,0.125},
    xticklabels={$2^{-6}$,$2^{-5}$,$2^{-4}$,$2^{-3}$},
    ylabel style={yshift=-5pt}, % move closer to y-axis ticks
    xlabel style={yshift=-5pt}  % move farther from x-axis ticks
]

%-----------------------------
% Deterministic FE REV
%-----------------------------
\addplot[
    color=unia-pink,
    line width=1.5pt,
    mark=square,
    mark size=2pt,
    every mark/.append style={solid},
] table[x=H, y=err_mean]{dat_files_Test2/err_l3_samples0.dat};

%-----------------------------
% Stochastic delta_1
%-----------------------------
\addplot[
    color=unia-orange,
    line width=1.5pt,
    mark=square,
    mark size=2pt,
    every mark/.append style={solid, fill=gray},
    error bars/.cd,
        y dir=both,
        y explicit
] table[
    x=H,
    y=err_mean,
    y error=err_std
]{dat_files_Test2/err_l3_samples_1e8.dat};

%-----------------------------
% Stochastic delta_2 with variable ±std error bars
%-----------------------------
\addplot[
    color=unia-lightblue,,
    line width=1.5pt,
    mark=square,
    mark size=2pt,
    every mark/.append style={solid, fill=gray},
    error bars/.cd,
        y dir=both,
        y explicit
] table[
    x=H,
    y=err_mean,
    y error=err_std
]{dat_files_Test2/err_l3_samples_1e9.dat};

\end{loglogaxis}

\node at (plot2.south) [yshift=-45pt] {(b) $\ell = 3$};

\begin{loglogaxis}[
    name = plot3, 
    width=0.3\textwidth,
    at={(0.66\textwidth,0\textwidth)},
    xmin=0.015625, xmax=0.0625,
    ymin=5e-3, ymax=1,
    xlabel={Coarse element size $H$},
    ylabel={Relative Error $\mathcal{E}_\mathcal{Q}$},
    ylabel near ticks,
    yticklabel pos=left,
    tick align=outside,
    tick pos=left,
    grid=both,
    minor y tick num=9,
    minor grid style={dotted, black!30},
    major grid style={line width=0.5pt, draw=black!50, dashed},
    legend style={at={(-2.2,1.8)}, anchor=north west, fill=none},
    legend cell align=left,
    enlargelimits=true,
    xtick={0.015625,0.03125,0.0625},
    xticklabels={$2^{-6}$,$2^{-5}$,$2^{-4}$},
    ylabel style={yshift=-5pt}, % move closer to y-axis ticks
    xlabel style={yshift=-5pt}  % move farther from x-axis ticks
]

%-----------------------------
% Deterministic FE REV
%-----------------------------
\addplot[
    color=unia-pink,
    line width=1.5pt,
    mark=square,
    mark size=2pt,
    every mark/.append style={solid},
] table[x=H, y=err_mean]{dat_files_Test2/err_l4_samples0.dat};

%-----------------------------
% Stochastic delta_1
%-----------------------------
\addplot[
    color=unia-orange,
    line width=1.5pt,
    mark=square,
    mark size=2pt,
    every mark/.append style={solid, fill=gray},
    error bars/.cd,
        y dir=both,
        y explicit
] table[
    x=H,
    y=err_mean,
    y error=err_std
]{dat_files_Test2/err_l4_samples_1e8.dat};

%-----------------------------
% Stochastic delta_2 with variable ±std error bars
%-----------------------------
\addplot[
    color=unia-lightblue,
    line width=1.5pt,
    mark=square,
    mark size=2pt,
    every mark/.append style={solid, fill=gray},
    error bars/.cd,
        y dir=both,
        y explicit
] table[
    x=H,
    y=err_mean,
    y error=err_std
]{dat_files_Test2/err_l4_samples_1e9.dat};

%-----------------------------
% Legend
%-----------------------------
\legend{Classical Implementation, Quantum Impl. $N_\text{samples} = 10^{8}$, Quantum Impl. $N_\text{samples} = 10^{9}$}
\end{loglogaxis}

\node at (plot3.south) [yshift=-45pt] {(c) $\ell = 4$};

\end{tikzpicture}
\caption{Relative errors \eqref{eq:err} for Test Case 2 between the LOD approximations and a reference solution  for $\ell=2,3,4$. Results are shown for both the classical LOD and its quantum variant with varying numbers of samples. The classical implementation follows \cref{algo:2}, the quantum implementation follows \cref{algo:4}. For the quantum variant, results are averaged over six realizations of the full procedure, and the corresponding error bars are shown.}
\label{fig:Error2}
\end{figure}

First, as shown in \cref{fig:Error1,fig:Error2}, the results from the classical LOD implementation exhibit the expected behavior for this choice of quantities of interest. Recall that this corresponds to the simplest variant of the LOD method (see, e.g., \cite{Moritz22}) and is consistent with the error estimations of \cref{thm:PGLOD-error}. The error is characterized by an initial increase as the coarse mesh size decreases, followed by a stagnation. Also, we observe the exponential decays of the error as~$\ell$ increases. However, the key aspect to analyze here is the deviation in the results between the quantum and classical approaches. For both test cases, we observe, that the error behavior is consistent with the findings in \cref{sec.Sens}, namely that the influence of quantum noise increases as $H$ decreases (i.e., the number of coarse patches increases) and as $\ell$ increases. Moreover, comparing the sample sizes $10^8$ and $10^9$ demonstrates that the number of samples, and hence the measurement accuracy, is critical for the overall performance of the method. \Cref{tab:tabError} reports the average relative measurement error for different sample sizes. \Cref{fig:Error1,fig:Error2} show that the macro-scale scheme amplifies quantum errors, in agreement with \eqref{eq:ErrCond}, emphasizing the need for sufficiently accurate measurements in the hybrid approach. As the number of measurements increases, the quantum results converge to the classical solution. Notably, even with $10^9$ samples corresponding to a relative measurement accuracy of around $10^{-4.5}$, we already obtain a good approximation.

\begin{table}[htbp]
\centering
\begin{tabular}{@{}c|cc|cc@{}}
 & \multicolumn{2}{c|}{Test Case 1} & \multicolumn{2}{c}{Test Case 2} \\ 
\midrule
Number of samples & $10^8$ & $10^9$ & $10^8$ & $10^9$ \\ 
\midrule
Avg. Relative Error - Measurement $\widehat{\s}$ & $9.7 \times 10^{-5}$ & $3.1 \times 10^{-5}$ & $1.1 \times 10^{-4}$ & $3.5 \times 10^{-5}$ \\ 
Avg. Relative Error - Measurement $\widehat{\mathfrak{m}}$ & $3.8 \times 10^{-5}$ & $1.2 \times 10^{-5}$ & $3.4 \times 10^{-5}$ & $1.1 \times 10^{-5}$ \\ 
\bottomrule
\end{tabular}
\caption{Average relative errors for measurements of the quantities $\widehat{\s}$ and $\widehat{\mathfrak{m}}$ as defined in \cref{algo:4} for Test Cases 1 and 2 with the two different numbers of samples considered.}
\label{tab:tabError}
\end{table}

\subsection{Discussion}

As shown, a major drawback of our methodology is that the macro-scale scheme amplifies quantum noise. This amplification is primarily linked to the construction of the LOD basis functions (i.e., larger patch sizes lead to increased noise) as well as to the conditioning of the global stiffness matrix. In particular, smaller values of $H$ result in a stronger amplification of the noise. Nevertheless, more advanced approaches, capable of both improving the approximation and mitigating the propagation of quantum noise, could be explored. These include improved constructions of the LOD basis functions and more effective truncation strategies. In addition, advanced variants such as the Super-Localized Orthogonal Decomposition (SLOD)~\cite{Moritz23} and the Hierarchical Super-Localized Orthogonal Decomposition (HSLOD)~\cite{Garay24} may be considered as they provide better conditioning compared to the classical method. However, these improvements come at the cost of a more complex quantum implementation, and their overall efficiency in this context remains uncertain. 

Furthermore, more sophisticated measurement strategies could be considered to improve accuracy while reducing the number of required samples \cite{Brassard02,Suzuki20}. In practice, these methods are not yet usable due to the increased circuit size, but in theory they would reduce the complexity to be linear in the inverse error tolerance.

\begin{table}
\centering
\begin{tabular}{@{}c|c|c|c@{}}
Method & Runtime (offline) & Runtime (online) & Memory \\
\midrule
Classical FEM & -- & $h^{-d} \log \tol^{-1}$ & $h^{-d}$ \\
Classical LOD & $2^d(\ell + 1)^dh^{-d} \log \tol^{-1}$ & $H^{-d} \log \tol^{-1}$ & $(H/h)^d + H^{-d}$ \\
Quantum FEM & -- & $\tol^{-1}\polylog h^{-1}$ & $d\log h^{-1}$ \\
Quantum LOD & $\operatorname{poly}(\ell, H) \tol^{-1} \polylog h^{-1}$ & $H^{-d} \log \tol^{-1}$ & $d \log (H/h) + H^{-d}$ \\ 
\bottomrule
\end{tabular}
\caption{Comparison of asymptotic complexity of different methods. Compared to the classical finite element method, the classical Localized Orthogonal Decomposition trades online runtime and memory requirements for a longer offline computation. The quantum methods trades the polynomial dependency in $h$ to a linear dependency in the error tolerance, while the required number of qubits are depend only logarithmically on the fine scale $h$. The proposed method improves over the monolithic quantum finite element approach, e.g.~\cite{Deiml25}, by moving all quantum computation offline.}
\label{tab:complexity}
\end{table}

Nonetheless, the preliminary results are already promising and suggest potential for future applications.
It must be emphasized that this quantum approach targets problems that are intractable for classical computers. Therefore, it is important to compare the complexity of this approach with that of the classical one. For this, note the asymptotic complexities given in \Cref{tab:complexity}. In particular, the quantum algorithm we proposed scales only logarithmically with the fine grid size $h$ and thus the fine-scale parameter~$\varepsilon$. This comes at the cost of a linear dependence on the measurement tolerance scaling with the overall error tolerance~$\tol$ in the runtime. For problems with small $\varepsilon$, but only limited accuracy requirements, the quantum method may then easily outperform classical approaches, at least in theory. This is even more pronounced with respect to memory requirements, where the $\tol^{-1}$ factor is not present for the quantum method, meaning that quantum algorithms might be applicable where classical methods simply cannot store the solution vector. The proposed algorithms enables the quantum computation to be outsourced offline, after which arbitrary right-hand sides can be considered. It offers a clear interface for the quantum solver, thereby addressing issues with state preparation for the right-hand sides of the quantum sub-procedure.

At present, existing quantum hardware cannot implement circuits of the size required by the proposed algorithm \cite{Deiml25}. This is mostly due to noise induced errors, which cannot be effectively corrected in current machines. Real world application of our approach thus requires further improvements in the hardware. 

\section{Conclusion}
In this research article, we have demonstrated the feasibility of a hybrid quantum--classical framework for numerical homogenization of scalar linear PDEs with rough coefficients. This hybrid approach not only overcomes the computational limitations of classical numerical homogenization, particularly the cost of computing local solutions, but also addresses some of the current limitations of quantum computing for large-scale PDEs. In this spirit, we strongly believe that combining multi-scale approaches with quantum computing is a promising path instead of solving a large PDEs with quantum but analyzing only quantities of interest defined on a coarse grid. Nevertheless, as previously discussed, several challenges must be addressed for this hybrid approach to be successfully applied to realistic, complex engineering problems.

\printbibliography

\appendix

\section{Domain extension}
\label{sec:DomainExtension}

\begin{figure}[H]
\centering	
\begin{tikzpicture}[scale=0.5]

% Grid parameters
\def\cells{12}    % number of cells per row/column
\def\size{1}      % size of one cell in cm

% Draw 12x12 grid
\draw[step=\size, line width=0.5pt, gray] (0,0) grid (\cells*\size, \cells*\size);

% Draw thicker vertical line at x=2 from y=2 to y=10
\draw[line width=2pt] (2*\size,2*\size) -- (2*\size,10*\size);
\draw[line width=2pt] (2*\size,10*\size) -- (10*\size,10*\size);
\draw[line width=2pt] (10*\size,10*\size) -- (10*\size,2*\size);
\draw[line width=2pt] (10*\size,2*\size) -- (2*\size,2*\size);

\draw[line width=1.5pt] (2*\size,0*\size) -- (2*\size,12*\size);
\draw[line width=1.5pt] (0*\size,10*\size) -- (12*\size,10*\size);
\draw[line width=1.5pt] (10*\size,12*\size) -- (10*\size,0*\size);
\draw[line width=1.5pt] (0*\size,2*\size) -- (12*\size,2*\size);

\draw[line width=1.5pt] (0*\size,0*\size) -- (0*\size,12*\size);
\draw[line width=1.5pt] (0*\size,12*\size) -- (12*\size,12*\size);
\draw[line width=1.5pt] (12*\size,12*\size) -- (12*\size,0*\size);
\draw[line width=1.5pt] (12*\size,0*\size) -- (0*\size,0*\size);

\node[black, scale=2] at (3,3) {$\Omega$};
\node[black, scale=1.5] at (1.2,1.2) {$\Omega^\ell$};

\node[black, scale=1] at (-3,1) {$\mathfrak{A}(-x_1,-x_2)$};
\draw[->, thick] (-0.5,1) -- (0.5,1);

\node[black, scale=1] at (-3,6) {$\mathfrak{A}(-x_1,x_2)$};
\draw[->, thick] (-0.5,6) -- (0.5,6);

\node[black, scale=1] at (-3,11) {$\mathfrak{A}(-x_1,2-x_2)$};
\draw[->, thick] (-0.5,11) -- (0.5,11);

\node[black, scale=1] at (6,13) {$\mathfrak{A}(x_1,2-x_2)$};
\draw[->, thick] (6,12.5) -- (6,11.5);

\node[black, scale=1] at (16,11) {$\mathfrak{A}(2-x_1,2-x_2)$};
\draw[->, thick] (12.5,11) -- (11.5,11);

\node[black, scale=1] at (15.5,6) {$\mathfrak{A}(2-x_1,x_2)$};
\draw[->, thick] (12.5,6) -- (11.5,6);

\node[black, scale=1] at (15.5,1) {$\mathfrak{A}(2-x_1,-x_2)$};
\draw[->, thick] (12.5,1) -- (11.5,1);

\node[black, scale=1] at (6,-1) {$\mathfrak{A}(x_1,-x_2)$};
\draw[->, thick] (6,-0.5) -- (6,0.5);

\node[black, scale=1] at (6,6.5) {$\mathfrak{A}(x_1,x_2)$};

\end{tikzpicture}
\caption{Extension of the domain $\Omega=[0,1]^2$ with the extended coarse mesh $\mathcal{G}_H^{\ell}$ and $\ell=2$. Additionally the extension of the diffusion coefficient $\mathfrak{A}$ is displayed (inspired from \cite{ThesisMohr}).}
\label{fig:extension}
\end{figure}
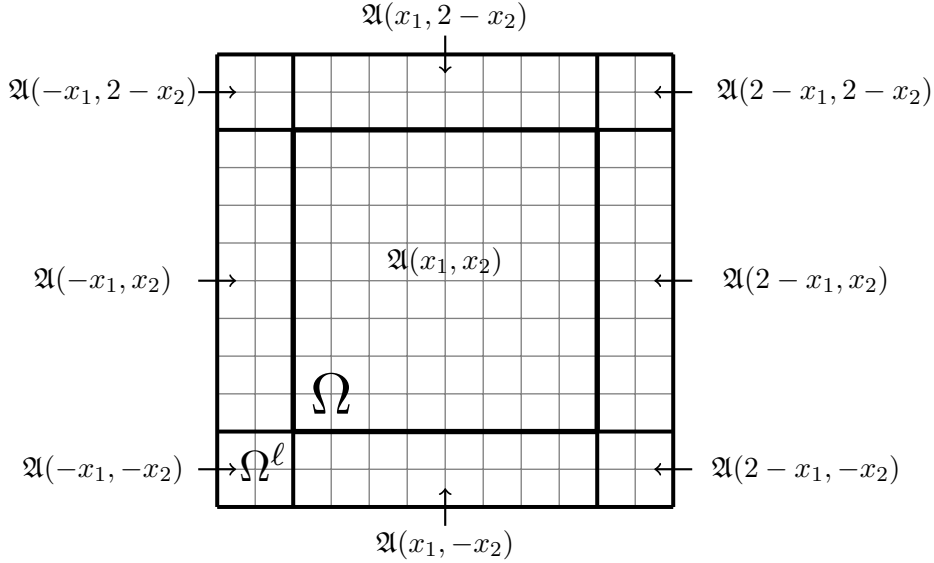

For nodes $j \in \mathcal{N}_H$ near the boundary $\partial \Omega$, the patch $\nei^\ell(\omega_j)$ may not be fully contained in the domain, i.e., 
$\nei^\ell(\omega_j) \not \subset \Omega$. In this case, for the ease of implementation, we use the technique introduced in \cite{ThesisMohr} and which allows to handle homogeneous Dirichlet boundary conditions. We extend the domain $\Omega=[0,1]^d$ by $\ell$-layers of coarse mesh elements. More precisely, we have 
\begin{equation*}
    \Omega^\ell:=\{x \in \mathbb{R}^d \ {\rm such \ that} \ \lvert x - \Omega\rvert_{\infty}<\ell H \}.
\end{equation*}
By $\mathcal{G}_H^\ell$ we denote the extension of the Cartesian coarse mesh 
$\mathcal{G}_H$ to $\Omega^\ell$. The diffusion coefficient $\mathfrak{A}$ is extended to the overlapping domain $\overline{\Omega^\ell}$ by mirroring it at the boundary of $\Omega$. See \cref{fig:extension} for an illustration in two dimensions. 
The extended diffusion coefficient is defined, in two dimensions for $(x_1,x_2) \in \mathbb{R}^2$, as 
 \begin{equation}
     \mathfrak{A}^{\rm ext} (x_1, x_2) = \mathfrak{A}(\tilde{x}_1, \tilde{x}_2)
    \label{eq:Aext}
 \end{equation}
 where
\begin{equation*}
   \tilde x_i =
\begin{cases}
x_i, & 0 \le x_i \le 1, \\
-\,x_i, & x_i < 0, \\
2 - x_i, & x_i > 1,
\end{cases}
\qquad i=1,2. 
\end{equation*}

Then, for $j \in \mathcal{N}_H$ such that $\nei^\ell(\omega_j) \not \subset \Omega$, we define the patch  $\nei^{\ell, \rm ext}(\omega_j)$ as previously, i.e., following \eqref{eq:Patch}, but using the extended coarse mesh $\mathcal{G}_H^\ell$ instead of $\mathcal{G}_H$. The local problems are then solved on $\nei^{\ell, \rm ext}(\omega_j)$, ignoring global Dirichlet boundary conditions but including the extended diffusion coefficient $\mathfrak{A}^{\rm ext}$ \eqref{eq:Aext}. 
We denote the resulting functions by $\tilde{\phi}^{\ell,\rm ext}_j \in H_0^1(\nei^{\ell, \rm ext}(\omega_j))$. To account for the global Dirichlet boundary conditions, we need to correct these functions. First, we extend
 $\tilde{\phi}^{\ell,\rm ext}_j$ by zero to $\mathbb{R}^d \setminus \Omega^\ell$. Then, in two dimensions, the functions $\tilde{\phi}^\ell_j \in H^1_0(\Omega)$ which satisfy global Dirichlet boundary conditions, are defined for $(x_1,x_2) \in \Omega$ as
 \begin{equation*}
 \begin{split}
 \tilde{\phi}^\ell_j(x_1, x_2) &=   \tilde{\phi}^{\ell,\rm ext}_j(x_1,x_2) - \tilde{\phi}^{\ell,\rm ext}_j(-x_1,x_2) - \tilde{\phi}^{\ell,\rm ext}_j(x_1,-x_2) \\
&- \tilde{\phi}^{\ell,\rm ext}_j(2-x_1,x_2) - \tilde{\phi}^{\ell,\rm ext}_j(x_1,2-x_2) + \tilde{\phi}^{\ell,\rm ext}_j(-x_1,-x_2) \\
&+ \tilde{\phi}^{\ell,\rm ext}_j(2-x_1,-x_2)
+ \tilde{\phi}^{\ell,\rm ext}_j(2-x_1,2-x_2) + \tilde{\phi}^{\ell,\rm ext}_j(-x_1,2-x_2).
 \end{split}
 \end{equation*}
As in our hybrid strategy, we only need quantum measurements. For a node $i \in \mathcal{N}_H$, defined by its coordinates $(i_1,i_2)$,  we define an extended nodal basis function  $\phi_{i=(i_1, i_2)}^{\rm ext}$ as 
 \begin{equation*}
 \begin{split}
\phi_{i=(i_1, i_2)}^{\rm ext} =&  \phi_{(i_1, i_2)} -\phi_{(-i_1, i_2)} - \phi_{(i_1, -i_2)} -\phi_{(2-i_1, i_2)} - \phi_{(i_1, 2-i_2)} \\
&+ \phi_{(-i_1, -i_2)} +\phi_{(2-i_1, -i_2)}
+ \phi_{(2-i_1, 2-i_2)} + \phi_{(-i_1, 2-i_2)}.
 \end{split}
 \end{equation*}
It follows then, for nodes $j \in \mathcal{N}_H$ near the boundary $\partial \Omega$, 
\begin{equation}
    a(\phi_i, \tilde{\phi}^\ell_j) = a(\phi_i^{\rm ext},  \tilde{\phi}_j^{\ell,\rm ext}). 
\label{eq:newhat}
\end{equation}
Then, the computation of \eqref{eq:newhat} follows the same way as in \cref{algo:2} or \cref{algo:4}. The extension of this procedure in three dimensions is straightforward.

\section{Reminder about sensitivity of linear systems}
\label{sec:Reminder}

In this section, we recall some classical results on the sensitivity of linear systems. As example, we consider the linear system
\begin{equation}
    Ax=b.
\label{eq:linsys}    
\end{equation}
 where $A\in \mathbb{R}^{n \times n}$, $n \in \mathbb{N}$, is a non-singular matrix, $b \in \mathbb{R}^n$ and the perturbed system
 \begin{equation}
    (A+\delta A)(x+\delta x)=b.
\label{eq:linsyspert}    
\end{equation}
In this section, we denote by $\lVert \, \bullet \, \rVert$ a matrix norm which is subordinate to some vector norm $\lVert \, \bullet \, \rVert$.

\begin{lemma}
    Let $B \in \mathbb{R}^{n\times n}$ such that $\lVert B \rVert<1$, then the matrix $I-B$ is non-singular, and it holds that 
    \begin{equation*}
        \lVert (I-B)^{-1} \rVert \leq \frac{1}{1-\lVert B \rVert}.
    \end{equation*}
\label{lemma:Sens1}
\end{lemma}
\begin{proof}
    For every $y \in \mathbb{R}^n$ such that $y\neq 0$, 
\begin{equation*}
        \lVert (I- B)y \rVert \geq \lVert y \rVert - \lVert B \rVert \lVert y \rVert = (1 - \lVert B \rVert) \lVert y\rVert > 0.
\end{equation*}
Therefore, the linear system
\begin{equation*}
    (I-B)y = 0
\end{equation*}
has a unique solution $y=0$, and $I-B$ is non-singular.
The estimate of the norm of the inverse of $I-B$ follows from
\begin{align*}
    1 &= \lVert (I-B)^{-1} (I-B)\rVert \\
     &= \lVert (I-B)^{-1} -(I-B)^{-1}B\rVert \\
     &\geq \lVert (I-B)^{-1} \rVert - \lVert (I-B)^{-1} B \rVert \\
     & \geq \lVert (I-B)^{-1} \rVert - \lVert (I-B)^{-1}\rVert \lVert B \rVert \\
    & \geq (1- \lVert B \rVert) \lVert (I-B)^{-1}\rVert . \qedhere
\end{align*}
\end{proof}

\begin{corollary}
    Let $A \in \mathbb{R}^{n \times n}$ be a non-singular matrix, and $\delta A \in \mathbb{R}^{n \times n}$, such that $$\lVert \delta A \rVert \lVert A^{-1} \rVert < 1.$$ Then $A + \delta A$ is non-singular, and it holds that 
\begin{equation*}
    \lVert (A + \delta A)^{-1} \rVert \leq \frac{\lVert A^{-1} \rVert}{1-\lVert A^{-1} \delta A\rVert} \leq \frac{\lVert A^{-1} \rVert}{1-\lVert A^{-1} \rVert \lVert \delta A\rVert} 
\end{equation*}
\label{cor:sens2}
\end{corollary}

\begin{proof}
Noting that $A + \delta A = A(I+ A^{-1}\delta A)$, the existence of $(A+ \delta A)^{-1}$ follows from \cref{lemma:Sens1}, since 
\begin{equation*}
   \lVert A^{-1}\delta A \rVert \leq \lVert \delta A \rVert \lVert A^{-1} \rVert < 1.
\end{equation*}
Finally, we have 
\begin{equation*}
   \lVert (A+\delta A)^{-1} \rVert \leq \frac{\lVert A^{-1} \rVert}{1-\lVert A^{-1} \delta A \rVert} \leq \frac{\lVert A^{-1} \rVert}{1-\lVert A^{-1} \rVert \lVert\delta A \rVert}. \qedhere
\end{equation*}
\end{proof}
\Cref{cor:sens2} shows that for a non-singular matrix $A$, the perturbed matrix $A+\delta A$ is also non-singular if the perturbation $\delta A$ is sufficiently small.
\begin{theorem}
    Let $A \in \mathbb{R}^{n \times n}$, $\delta A \in \mathbb{R}^{n \times n}$, such that $A$ is non-singular, and assume that $\lVert A^{-1} \rVert  \lVert \delta A\rVert <1 $.  Let $x$  and $x + \delta x$ be the solution of the linear system \eqref{eq:linsys} and the perturbed system \eqref{eq:linsyspert}, respectively, then the following estimation of the relative error holds
\begin{align*}
        \frac{\lVert \delta x\rVert}{\lVert x \rVert}  \leq \frac{\kappa(A) }{1-\kappa(A)  \frac{\lVert \delta A \rVert}{\lVert A \rVert } } \frac{\lVert \delta A \rVert}{\lVert A \rVert} .
\end{align*}
\label{thm:PertubedSystem}
\end{theorem}
\begin{proof}
    The assumption $\lVert A^{-1} \rVert  \lVert \delta A\rVert <1$, ensures that $A + \delta A$ is non-singular according to~\cref{cor:sens2}. Then, solving \eqref{eq:linsyspert} for $\delta x$ leads to 
\begin{align*}
    \delta x &= -(A+ \delta A)^{-1}\delta A x.
\end{align*}
Then, using the same strategy as in the proof of \cref{cor:sens2}, it follows that
\begin{equation*}
    \frac{\lVert \delta x\rVert}{\lVert x \rVert} \leq \lVert (I+A^{-1} \delta A)^{-1} \rVert \lVert A^{-1} \rVert  \lVert \delta A \rVert.
\end{equation*}
\Cref{cor:sens2} leads to
\begin{align*}
        \frac{\lVert \delta x\rVert}{\lVert x \rVert} &\leq \frac{\lVert A^{-1} \rVert}{1-\lVert A^{-1}\rVert \lVert \delta A \rVert } \lVert \delta A \rVert\\
        & \leq \frac{\lVert A^{-1} \rVert \lVert A \rVert  }{1-\lVert A^{-1}\rVert  \lVert A \rVert  \frac{\lVert \delta A \rVert}{\lVert A \rVert } } \frac{\lVert \delta A \rVert}{\lVert A \rVert }.
\end{align*}
\end{proof}
In addition, for sufficiently small perturbations, it holds, 
\begin{equation*}
  \frac{\kappa(A) }{1-\kappa(A)  \frac{\lVert \delta A \rVert}{\lVert A \rVert } } \frac{\lVert \delta A \rVert}{\lVert A \rVert } \approx   \kappa(A)  \frac{\lVert \delta A \rVert}{\lVert A \rVert } + \mathcal{O} \left( \left(   \frac{\lVert \delta A \rVert}{\lVert A \rVert }  \right )^2  \right).
\end{equation*}
Hence, the standard first-order perturbation estimate is given by 
\begin{equation}
        \frac{\lVert \delta x\rVert}{\lVert x \rVert}  \approx \kappa(A) \frac{\lVert \delta A \rVert}{\lVert A \rVert }.
\label{eq:firstorder}
\end{equation}
\end{document}